\documentclass[12pt]{article}
\usepackage{graphicx} 

\usepackage{fullpage,comment}
\usepackage{amssymb}%
\usepackage{amsmath}%
\usepackage{amsfonts}%
\usepackage{amsthm}
\usepackage{changepage}
\usepackage{enumitem}
\usepackage{dsfont}
\usepackage{thm-restate}
\usepackage{xcolor}
\usepackage[normalem]{ulem}
\newtheorem{theorem}{Theorem}[section]
\newtheorem{proposition}[theorem]{Proposition}

\newtheorem{lemma}[theorem]{Lemma}
\newtheorem{corollary}[theorem]{Corollary}

\newtheorem{question}[theorem]{Question}

\newtheorem*{lemma*}{Lemma}
\newtheorem{claim}[theorem]{Claim}
\newtheorem{obs}[theorem]{Observation}
\newtheorem*{claim*}{Claim}
\newtheorem{problem}[theorem]{Problem}
\newtheorem{construction}[theorem]{Construction}

\theoremstyle{remark}
\newtheorem*{remark}{Remark}

\newcommand{\bb}[1]{\mathbb{#1}}
\newcommand{\ca}[1]{\mathcal{#1}}
\newcommand{\bc}{,\allowbreak}
\newcommand{\ff}[2]{\left\lfloor\frac{#1}{#2}\right\rfloor}

\newcommand{\mr}[1]{\mathrm{#1}}
\newcommand{\cf}[2]{\left\lceil\frac{#1}{#2}\right\rceil}

\DeclareMathOperator{\sat}{sat}
\DeclareMathOperator{\prsat}{sat^\ast}
\DeclareMathOperator{\ssat}{ssat}
\DeclareMathOperator{\Sat}{Sat}
\DeclareMathOperator{\Prsat}{Sat^\ast}

\DeclareMathOperator{\ex}{ex}

\setenumerate[0]{label=(\arabic*)}

\title{Proper rainbow saturation for trees}
\author{Andrew Lane\footnotemark[1]\thanks{Department of Mathematics and Statistics, University of Victoria, Canada. \\ Email: \texttt{\{andrewlane,nmorrison\}@uvic.ca}. } \thanks{Research supported by the Jamie Cassels Undergraduate Research Awards and Science Undergraduate Research Awards, University of Victoria.} 
    \and Natasha Morrison\footnotemark[1] \thanks{Research supported by NSERC Discovery Grant RGPIN-2021-02511 and NSERC Early Career Supplement DGECR-2021-00047.}
	}
\date{\today}

\begin{document}

\maketitle

\begin{abstract}
Given a graph $H$, we say that a graph $G$ is \emph{properly rainbow $H$-saturated} if: (1) There is a proper edge colouring of $G$ containing no rainbow copy of $H$; (2) For every $e \notin E(G)$, every proper edge colouring of $G+e$ contains a rainbow copy of $H$. The \emph{proper rainbow saturation number} $\prsat(n,H)$ is the minimum number of edges in a properly rainbow $H$-saturated graph. In this paper we initiate a systematic study of the proper rainbow saturation number for trees. We obtain exact and asymptotic results on $\prsat(n,T)$ for several infinite families of trees. Our proofs reveal connections to the classical saturation and semi-saturation numbers.
\end{abstract}

\section{Introduction}
Given a graph $H$, say a graph $G$ contains a copy of $H$ if $G$ has a subgraph isomorphic to $H$; say $G$ is \emph{$H$-free} otherwise. We say that $G$ is \emph{$H$-saturated} if $G$ is $H$-free, but adding any edge to $G$ creates a copy of $H$.
Given a number $n \in \bb{N}$ and a graph $H$, the \emph{saturation number} $\sat(n,H)$ is the minimum number of edges in an $H$-saturated graph on $n$ vertices.
Since a maximal $H$-free graph is also $H$-saturated, the Tur\'an extremal number $\ex(n,H)$, which counts the maximum number of edges in an $H$-free graph on $n$ vertices, also counts the maximum number of edges in an $H$-saturated graph on $n$ vertices.
Thus, the saturation number is a natural dual to the extremal number.

Study of the saturation number was initiated by Erd\"os, Hajnal, and Moon~\cite{ErdosHajnalMoon1964} in 1964.
Since then it has been widely studied, along with many variants; see the survey by Currie, J. Faudree, R. Faudree, and Schmitt~\cite{faudree2011survey}.

This paper concerns an edge-colouring variant of the saturation number.
An \emph{edge-colouring} of a graph $G$ is a function $\phi:E(G)\to C$, where $C$ is an arbitrary set of so-called \emph{colours}.
An edge-colouring $\phi$ of a graph $G$ is \emph{proper} if, for all $e,f \in E(G)$, $e\cap f\neq \emptyset$ implies $\phi(e)\neq \phi(f)$.
An edge-colouring is \emph{rainbow} if every edge receives a different colour, i.e., it is injective.
For graphs $G$ and $H$, and an edge-colouring $\phi$ of $G$, a \emph{rainbow copy of $H$} in $G$ under $\phi$ is a subgraph $H'$ of $G$ isomorphic to $H$ such that the restriction of $\phi$ to $E(H')$ is rainbow.

Given a graph $H$, we say that a graph $G$ is \emph{properly rainbow $H$-saturated} if the following hold:
\begin{enumerate}
    \item There is a proper edge colouring of $G$ containing no rainbow copy of $H$;
    \item For every $e \notin E(G)$, every proper edge colouring of $G+e$ contains a rainbow copy of $H$.
\end{enumerate}
The \emph{proper rainbow saturation number} $\prsat(n,H)$ is the minimum number of edges in a properly rainbow $H$-saturated graph. This was introduced in 2022 by Bushaw, Johnston, and Rombach~\cite{Bushaw2022} as a natural analogue of the  rainbow extremal number, introduced by Keevash, Mubayi, Sudakov, and Verstra\"ete~\cite{keevash2007rainbow} in 2007.
While $\prsat(n,H)$ is referred to in~\cite{Bushaw2022} as the rainbow saturation number, since a function termed the rainbow saturation number has already been studied in the literature (see, e.g., \cite{barrus2017colored}, \cite{behague2024rainbow}, \cite{girao2020rainbow}), we call $\prsat(n,H)$ the proper rainbow saturation number instead.
A straightforward application of a result of K\'azsonyi and Tuza~\cite{Kaszonyi1986} shows that $\prsat(n, H) = O(n)$ for all $H$~\cite{pcbjr,pcbj}. We show in~\cite{other} $\prsat(n,H) = \Omega(n)$ whenever $H$ has no isolated edge. So in fact $\prsat(n,H) = \Theta(n)$ for all connected $H$.

In this paper we focus our attention on the proper rainbow saturation number of trees. In particular, we provide asymptotically tight bounds for several infinite families of trees. Prior to this work, asymptotically tight bounds on $\prsat(n,H)$ were only known for $H = C_4$ (Halfpap, Lidický, and Masařík~\cite{halfpap2024proper}),  and $H = P_4$ (Bushaw, Johnston, and Rombach~\cite{halfpap2024proper}). 

Typically, upper bounds on saturation numbers (both rainbow and classical) are proved through finding an appropriate construction. Lower bounds are notoriously more difficult. Let the \emph{diameter} of a graph $G$ be $\max\{d_G(u,v): u,v \in V(G)\}$, where $d_G(u,v)$ denotes the \emph{distance} between $u$ and $v$, that is, the number of edges in a shortest path with endpoints $u$ and $v$. Our first theorem provides a general lower bound on the proper rainbow saturation number for connected graphs with diameter at least $5$.

\begin{restatable}{theorem}{genlower} \label{thm:gen_path_lower_bound}
Let $H$ be a connected graph with diameter at least $5$.
Then for all $n \in \bb{N}$, $\prsat(n,H)\ge n-1$.
\end{restatable} 
This is best possible, as evidenced by Theorem~\ref{thm:caterpillar} below. 

In fact, we can also achieve the lower bound of $n-1$ for a class of trees with diameter 4. 
For $k,m \ge 1$, let $B_{k,m}$ be the tree obtained by appending $m$ pendant vertices to an endpoint of the path $P_k$.
We refer to $B_{k,m}$ as a \emph{broom of length $k$}.

\begin{restatable}{theorem}{broom4} \label{thm:broom4}
For all $m \in \bb{N}$ and $n \ge 3(m+2)$, 
\[n-1 \le \prsat(n,B_{4,m}) \le 3(m+2)\left\lfloor \frac{n}{3(m+2)} \right\rfloor + \binom{n-3(m+2)\left\lfloor \frac{n}{3(m+2)} \right\rfloor}{2}.\]
\end{restatable}
Note that, for fixed $m$, $\prsat(n,B_{4,m}) = n+O(1)$ and, for infinitely many $n$, we have $\prsat(n,B_{4,m}) = n-1$.

Say a tree $T$ is a \emph{caterpillar} if the tree obtained by deleting all leaves in $T$ is a path (call this path the \emph{central path}).
Note that all brooms (and thus all stars and paths) are caterpillars.
\begin{restatable}{theorem}{caterpillar} \label{thm:caterpillar}
Let $\ell\ge 4$ and $k \ge \ell+2$, and let $T_{k,\ell}$ be a $k$-vertex caterpillar with central path $v_1,\ldots,v_\ell$.
Suppose $v_\ell$ has degree $2$ in $T_{k,\ell}$.
Then for all $n\ge (k+1)2^{\ell-2}$,
\[\prsat(n,T_{k,\ell}) \le n+(\ell-3)2^{\ell-3}.\]
\end{restatable}
Along with Theorem~\ref{thm:gen_path_lower_bound} this shows that $\prsat(n,T_{k,\ell}) = n + O(1)$ for any $\ell \ge 4$ and $k \ge \ell + 2$.
In particular, we have the following corollary.
\begin{corollary} \label{cor:path}
For all $k \ge 5$ and $n \ge (k+1)2^{k-4}$,
\[ n-1 \le \prsat(n,P_k) \le \begin{cases}
    9\ff{n}{9}+\binom{n-9\left\lfloor \frac{n}{9} \right\rfloor}{2}, & k=5, \\
    n+(k-5)2^{k-5}, & k\ge6.
\end{cases} \]
In particular, $\prsat(n,P_k)=n+O(1)$ for all $k \ge 5$, and $\prsat(n,P_5)=n-1$ for infinitely many $n$.
\end{corollary}

In light of Theorem~\ref{thm:caterpillar}, one may wonder whether the condition on the degree of $v_{\ell}$ in $T_{k,\ell}$ is necessary, or simply an artefact of the proof. The following theorem resolves this question (in a stronger sense) and tells us that it would not be possible to prove an upper bound of the form in Theorem~\ref{thm:caterpillar} with $d(v_{\ell}) > 2$. 

\begin{restatable}{theorem}{catcon}
\label{thm:cat_converse}
Let $T$ be a tree with diameter $\ge 4$. 
Suppose that every vertex of degree $2$ in $T$ has no leaf neighbours.
Let $w_0,\ldots,w_t$ be a longest path in $T$ such that the internal vertices $w_1,\ldots,w_{t-1}$ each have degree $2$ in $T$.
Let $r:= \max\{t,2\}$.
Then 
\[\prsat(n,T)\ge \left(1+\frac{1}{12r+52}\right)n+O(1). \]
\end{restatable}

We also provide an exact value for the proper rainbow saturation number for a different infinite family of trees, including $P_4$ (see Theorem~\ref{thm:Tkast rainbow sat}, below). We do this via revealing a connection to the classical saturation number. We generalize these results further (see Theorem~\ref{thm:double_star_sat_improved}) to so-called \emph{double stars}.

The paper is organised as follows. Theorem~\ref{thm:gen_path_lower_bound} is proved in Section~\ref{sec:gen}, Theorem~\ref{thm:broom4} in Section~\ref{sec:broom}, Theorem~\ref{thm:caterpillar} in Section~\ref{sec:cat}, and Theorem~\ref{thm:cat_converse} in Section~\ref{sec:catcon}. We prove results on double stars and subdivided stars in Section~\ref{sec:stars} before concluding with some open questions in Section~\ref{sec:concl}.

\subsection{Notation and terminology}
``Colourings" will always refer to edge-colourings.
For a graph $G$ and a colouring $\phi$ of $G$, we will refer to the pair $(G,\phi)$, called a \emph{coloured graph}.
A rainbow copy of a graph $H$ in $(G,\phi)$ is a rainbow copy of $H$ in $G$ under the colouring $\phi$.
This notation will be useful when many related graphs and colourings are involved.

For $n,m \in \bb{N}$ with $n \le m$, we write $[n,m]:=\{n,n+1,\ldots,m\}$ and $[n]:=[1,n]$.
For a graph $G$, a vertex $v \in V(G)$, and a set $A\subseteq V(G)$, denote $vA:= \{vu:u \in A\}$.
For a function $\phi:E\to C$ and a subset $E'\subseteq E$, let $\phi(E'):= \{\phi(e):e \in E'\}$, and let $\phi|_{E'}:E'\to C$ be the restriction of $\phi$ to $E'$.

For a graph $G$, let $\overline{G}$ denote the complement of $G$.
For $v \in V(G)$, let $G-v$ denote the graph with vertex set $V(G)\setminus\{v\}$ and edge set $E(G)\setminus\{vu:u \in V(G)\}$.
For $e \in E(G)$, let $G-e$ denote the graph with vertex set $V(G)$ and edge set $E(G)\setminus\{e\}$.
For graphs $G$ and $H$ with $V(G)\cap V(H)=\emptyset$, let $G+H$ denote the graph join of $G$ and $H$, which has vertex set $V(G)\cup V(H)$ and edge set $E(G)\cup E(H) \cup \{uv:u \in V(G),v \in V(H)\}$.
Let $G\cup H$ denote the disjoint union of $G$ and $H$, which has vertex set $V(G)\cup V(H)$ and edge set $E(G) \cup E(H)$.

\section{General lower bound for graphs containing $P_6$}~\label{sec:gen}
In this section, we prove Theorem~\ref{thm:gen_path_lower_bound}, restated here for convenience.
\genlower*

We denote by $\Prsat(n,H)$ the set of $n$-vertex properly rainbow $H$-saturated graphs with $\prsat(n,H)$ edges. For a tree $T$, given $u,v \in V(T)$, let $P_{uv}$ denote the unique $u$-$v$ path in $T$. For a graph $G$, given $u,v \in V(G)$.

\begin{proof}[Proof of Theorem~\ref{thm:gen_path_lower_bound}]
    We will show that no tree on $n\ge 3$ vertices is properly rainbow $H$-saturated.
    This suffices, as a disconnected graph in $\Prsat(n,H)$ has at most one tree component, which is necessarily a $K_1$ or $K_2$ (any larger tree is not properly rainbow $H$-saturated, and adding an edge between two $K_1$ or $K_2$ components could not yield a copy of $H$).
    
    Suppose, in order to obtain a contradiction, that for some $n\ge 3$, $T$ is an $n$-vertex tree that is properly rainbow $H$-saturated.
    Let $\phi$ be a proper colouring of $T$ containing no rainbow $H$.
    We will recolour $T$ and add an edge $e$ to $T$ in a specific colour so that the resulting colouring of $T+e$ is still proper and rainbow $H$-free, contradicting $T$ being properly rainbow $H$-saturated.

    Choose a root $r$ for $T$.
    Let $v_1$ be a vertex with $d(v_1)\ge 3$ of maximal distance from $r$, so that every vertex $x$ with $d(x,r)>d(v_1,r)$ has $d(x)\le 2$.
    For each $v \in V(T)$, let $T_v$ denote the subtree of $T$ rooted at $v$, let $f(v)$ be the parent of $v$ if $v\neq r$, and let $f(v)=v$ if $v=r$.
    Let $v_2,\ldots,v_k$ be the siblings of $v_i$, and let $y=f(v_1)$.
    For each $i \in [m]$, let $T_i:= T_{v_i}$.
    
    \begin{claim} \label{clm:gen_path_lower}
        Each $T_i$ is either a path with endpoint $v_i$, or $v_i$ has exactly one leaf neighbour and is the unique vertex of $T_i$ that has degree 3 in $T$.
    \end{claim}
    \begin{proof}[Proof of Claim]
        Let $x_1,\ldots,x_t$ be the children of $v_i$, and for each $j \in [t]$, let $c_j:= \phi(y_ix_j)$.
        Since every descendant $x$ of $v_i$ has $d(x,r) < d(v_i,r) = d(v_1,r)$, $T_{x_j}$ is a path rooted at $x_j$ for all $i \in [t]$; let $\ell_j$ be the number of vertices in $T_{x_j}$, and let $u_j$ be the leaf in $T_{x_j}$.
        If $v_i \neq y$, let $c_0:= \phi(v_iy)$, and let $c_0$ be a colour not in $\{c_1,\ldots,c_t\}$ otherwise.
        Let $T'$ be the subtree of $T$ induced by $V(T)\setminus \bigcup_{j=1}^t V(T_{x_j})$.

        By the hypothesis of the claim, we can assume that $t\ge 2$ and either $\ell_1=\ell_2=1$ or $\ell_1\ge \ell_2 \ge 2$.
        Let $\phi'$ be obtained from $\phi$ by properly recolouring each $T_{x_j}$ in alternating colours $c_0,c_j$.
        Extend $\phi'$ to a colouring of $T+u_1u_2$ by properly colouring $vu$ greedily with the first available colour in $c_0,c_1,c_2$.

        Every rainbow $P_6$ in $(T+u_1u_2,\phi')$ is contained entirely in $(T',\phi')=(T',\phi)$.
        To see this, first note that if such a path $P$ is contained entirely in $T_{i}$, it passes through $v_i$ at most once and thus intersects at most two components of $(T_{i}+u_1u_2)-v_i$, say the components containing $x_j,x_{j'}$ for some $2\le j,j' \le t$ (note that $x_1$ and $x_2$ are in the same component of $(T_{i}+v_1v_2)-v_i$).
        But then $P$ contains only colours in $\{c_0,c_1,c_j,c_{j'}\}$, so it is not rainbow.
        So we can assume that $P$ contains the path $z,x_j,v_i,y$ for some $z$ in $T_{i}$ and $j \in [t]$.
        If $zx_j = u_1u_2$, then since at least one of $u_1,u_2$ is adjacent to $v_i$, we are in the case that $\ell_1=\ell_2=1$, so $\phi(zx_j) = c_0 = \phi(v_iy)$.
        Otherwise, since $T_{x_j}$ receives an alternating colouring in colours $c_0,c_j$, we again have $\phi'(zx_j) = c_0 = \phi'(v_iy)$.
        Thus, $P$ is not rainbow.

        Suppose $(T+u_1u_2,\phi')$ has a rainbow copy $H'$ of $H$.
        Then $H'$ has a longest path $P=w_1,\ldots,w_m$, where $m \ge 6$.
        If $H'$ is contained in $T'$, then $H'$ is coloured by the original rainbow $H$-free colouring $\phi$ and is thus not rainbow.
        If $P$ contains an edge not in $T'$, then $P$ is not rainbow by the paragraph above, so $H'$ is not rainbow.
        Otherwise, neither of the endpoints of $P$ are in $T_{i}$ and $H'$ contains an edge $x_jv$ for some $j \in [t]$ and $v \in V(T_i)$.
        Let $\ell \in [m]$ such that $w_\ell$ is a vertex in $P$ of minimal distance to $y_i$.
        Then the path $P_{x_jw_\ell}$ contains the path $x_j,v_i,y$ and thus has length at least $2$, and either $p_1:=w_\ell,w_{\ell-1},\ldots,w_1$ or $p_2:=w_\ell,w_{\ell+1},\ldots,w_k$ has length at least $\lceil m/2\rceil \ge 3$, so $v,P_{x_jw_\ell},p_1$ or $v,P_{x_jw_\ell},p_2$ is a path with at least $6$ vertices that is contained in $H'$ but not in $T'$, so this path is not rainbow, and so $H'$ is not rainbow, a contradiction.
        This proves the claim.
    \end{proof}
    Since $d(v_1) \ge 3$ and $v_1$ has at most $2$ children by Claim \ref{clm:gen_path_lower}, we must have $v_1 \neq r$, i.e., $y \neq v_1$.
    For each $i \in [k]$, let $c_i:= \phi(yv_i)$, and let $c_0:=\phi(yf(y))$ if $y\neq r$ or $c_0 \notin \{c_1,\ldots,c_m\}$ otherwise.
    For each $i \in [k]$, let $u_i$ be the leaf of the maximal-length path rooted at $y_i$, and let $x_i$ be the neighbour of $v_i$ in $P_{v_iu_i}$ (if it exists).
    Assume without loss of generality that $d(v_2,u_2) \ge d(v_3,u_3)\ge \cdots \ge d(v_m,u_m)$.
    Let $T'$ be the subtree of $T$ induced by $V(T)\setminus\bigcup_{i=1}^m V(T_i)$.
    We will define the colouring $\phi'$ of $T$ based on $\phi$ as follows.
    
    For each $i\in [k]$, swap the colours $c_0$ and $\phi(v_ix_i)$ wherever they appear in $T_i$.
    Since $\phi(yv_i) = c_i\neq c_0$, this preserves the proper colouring.
    
    Also, this recolouring does not create any rainbow copies of $P_6$, since any $P_6$ affected by this recolouring contains two edges of the colour $c_0$.
    Now, for each $i \in [k]$, recolour $P_{v_iu_i}$ with alternating colours $c_0,c_j$; again, this does not create any rainbow copies of $P_6$.
    Let $\phi'$ be the resulting colouring of $T$ from these recolourings.
    Note that $\phi'|_{E(T')}=\phi|_{E(T')}$ and $(T_y,\phi'|_{T_y})$ is rainbow $P_6$-free.

    \begin{figure}
        \centering
        \includegraphics{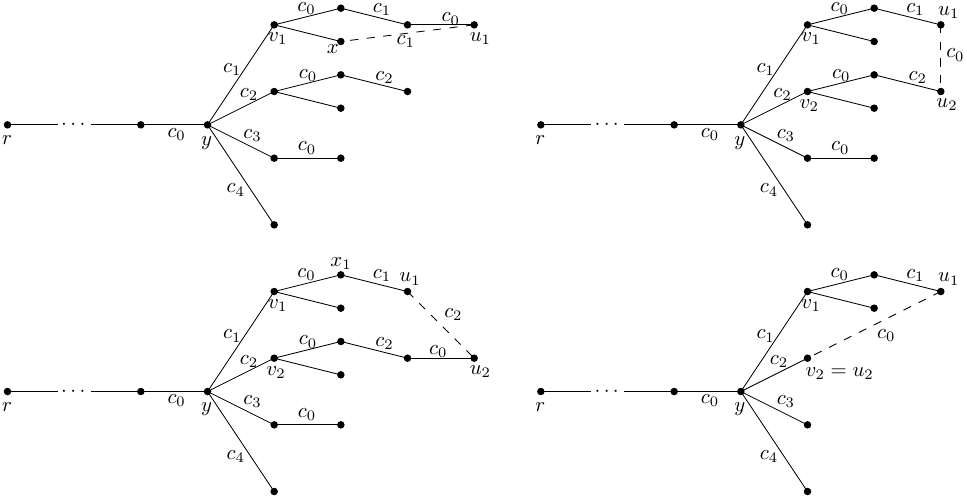}
        \caption{The coloured graph $(T+e,\phi')$ from the proof of Theorem~\ref{thm:gen_path_lower_bound}}
        \label{fig:gen_lower_end}
    \end{figure}
    
    Since $d(v_1) \ge 3$, by Claim~\ref{clm:gen_path_lower}, $v_1$ has a leaf neighbour $x$ and $d(v_1,u_1) \ge 2$.
    If $\phi'(u_1f(u_1)) = c_0$, then extend $\phi'$ to a proper colouring of $T+u_1x$ by setting $\phi(u_1x) = c_1$.
    Note that $\phi(v_1x) \neq \phi(v_1y) = c_1$, so this colouring is indeed proper.
    For any copy $P$ of $P_6$ containing $u_1x$, $P$ is either contained entirely in $T_1$, which is $3$-coloured with colours $c_0,c_1,\phi(v_1x)$, or $P$ contains the two edges $u_1x,v_1y$ of colour $c_1$ (see Figure~\ref{fig:gen_lower_end}, top left).
    So every rainbow copy of $P_6$ in $(T+u_1x,\phi')$ is in $(T',\phi)$ in this case.

    Now suppose that $\phi'(u_1f(u_1)) = c_1$.
    In this case, extend $\phi'$ to a proper colouring of $T+u_1u_2$ by colouring $u_1u_2$ with the first available colour in $c_0,c_2$. Any copy of $P_6$ containing $u_1u_2$ has a repeated colour in $\{c_0,c_1,c_2\}$, so every rainbow copy of $P_6$ in $(T+u_1u_2,\phi')$ is in $(T',\phi)$ in this case (see Figure~\ref{fig:gen_lower_end}).

    Thus, in all cases, there exists $e \in E(\overline{T})$ and an extension of $\phi'$ to $T+e$ such that every rainbow copy of $P_6$ in $(T+e,\phi')$ is in $(T',\phi)$.
    As in the last paragraph of the proof of Claim~\ref{clm:gen_path_lower}, it follows that $(T+e,\phi')$ is rainbow $H$-free.
    Thus, $T$ is not properly rainbow $H$-saturated.
    This completes the proof of Theorem~\ref{thm:gen_path_lower_bound}.
\end{proof}

\section{Length-$4$ brooms}\label{sec:broom}
In this section, we prove Theorem~\ref{thm:broom4}.
We prove the upper and lower bounds on $\prsat(n,B_{4,m})$ from Theorem~\ref{thm:broom4} separately in Theorems~\ref{thm:broom4_lower_bound} and~\ref{thm:broom4_upper} below, respectively.

In order to prove the lower bound on $\prsat(n,B_{4,m})$ in Theorem~\ref{thm:broom4}, we have the following lemma.
\begin{lemma} \label{lem:broom4_containing_tree}
Let $m \in \bb{N}$, and let $G$ be a graph with a a rainbow $B_{4,m}$-free proper colouring $\phi$.
Any path $v_0,v_1,v_2,v_3$ in $G$ with $d(v_0) \ge m+3$ has $\phi(v_0v_1) = \phi(v_2v_3)$.
In particular, for any $v_0,v_2 \in V$ such that $d(v_0,v_2) = 2$, if $d(v_0)\ge m+3$, then $d(v_2) \le 2$.
\end{lemma}
\begin{proof}
Take any proper colouring $\phi$ of $G$. If $\phi(v_0v_1) \neq \phi(v_2v_3)$, then $\phi(v_0v_1) \neq \phi(v_1v_2)$ and $d(v_0)-1\ge m+2$ implies that there exist $u_1,\ldots,u_m \in N(v_0)\setminus\{v_1\}$ such that $u_1v_0,u_2v_0,\ldots,u_mv_0$ do not receive the colours $\phi(v_1v_2),\phi(v_2v_3)$.
Thus, the edges $u_1v_0\bc u_2v_0\bc \ldots\bc u_mv_0$ and $v_0v_1\bc v_1v_2\bc v_2v_3$ form a rainbow copy of $B_{4,m}$.
This contradicts the fact that $G$ is properly rainbow $B_{4,m}$-saturated.
Thus, $\phi(v_0v_1)=\phi(v_2v_3)$.

Now, $v_2$ has at most one neighbour $v_3$ such that $\phi(v_2v_3)=\phi(v_0v_1)$, so $d(v_2)\le 2$.
\end{proof}

\begin{theorem} \label{thm:broom4_lower_bound}
For all $m \in \bb{N}$ and $n \ge 3$, $\prsat(B_{4,m})\ge n-1$.
\end{theorem}
The main ideas in the proof of Theorem~\ref{thm:broom4_lower_bound} are the same as in the proof of Theorem~\ref{thm:gen_path_lower_bound}, but the case analysis is slightly more complicated because in general, shorter rainbow paths are more difficult to avoid than longer rainbow paths.
As in the proof of Theorem~\ref{thm:gen_path_lower_bound}, for a tree $T$ and $u,v \in V(T)$, let $P_{uv}$ denote the unique $u$-$v$ path in $T$.
 
\begin{proof}
    We will show that no tree on $n \ge 3$ vertices is properly rainbow $B_{4,m}$-saturated. This suffices, as a disconnected graph in $\Prsat(n,B_{4,m})$ has at most one tree component, which is necessarily a $K_1$ or $K_2$ (any larger tree is not $P_5$-saturated and adding an edge between two $K_1$ or $K_2$ components could not yield a copy of $B_{4,m}$). 

    Suppose, in order to obtain a contradiction, that for some $n \ge 3$, $T$ is an $n$-vertex tree that is properly rainbow $B_{4,m}$-saturated. Let $\phi$ be a proper colouring of $T$ containing no rainbow $B_{4,m}$.
    As in the proof of Theorem~\ref{thm:gen_path_lower_bound}, we will recolour $T$ and add an edge $e$ to $T$ in a specific colour so that the resulting colouring of $T+e$ is still proper and rainbow $B_{4,m}$-free, contradicting $T$ being properly rainbow $B_{4,m}$-saturated.
    Observe that $T$ contains at least one vertex of degree at least $m+3$ (else the graph obtained by adding an edge between two leaves of $T$ can be $(m+2)$-edge-coloured and hence $T$ is not properly rainbow $B_{4,m}$-saturated).

    \begin{claim}\label{cl:fewleaves1broom}
        Let $v \in T$ have degree $d \ge 3$. Then at most $d-2$ neighbours of $v$ are leaves. 
    \end{claim}
    \begin{proof}
        Suppose $v$ has degree $d \ge 3$ and $d-1$ leaf neighbours. Let $x$ be a non-leaf neighbour of $v$ (which exists else $T$ is a star and $T + e$ contains no $B_{4,m}$ for any $e$). Then adding an edge coloured $\phi(xv)$ between any two leaves incident to $v$ cannot create a rainbow $P_5$ as it must use the edge $xv$ and would then have a repeated colour. Thus, since $B_{4,m}$ contains $P_5$, adding such an edge cannot create a rainbow $B_{4,m}$.
    \end{proof}

    Choose a root $r$ for $T$. Let $y$ be a vertex with $d(y)\ge m+3$ of maximal distance from $r$, so that every vertex $x$ with $d(x,r)>d(y,r)$ has $d(x)\le 2$.
    For each $v \in V(T)$, let $T_v$ denote the subtree of $T$ rooted at $v$, let $f(v)$ be the parent of $v$ if $v\neq r$, and let $f(v)=v$ if $v=r$.
    Let $v_1,\ldots,v_k$ be the children of $y$.
    For each $i \in [k]$, let $T_i:= T_{v_i}$.
    
    \begin{claim}\label{cl:fewleaves2broom}
    Each $T_i$ is either a path with endpoint $v_i$, or $v_i$ has exactly one leaf neighbour and is the unique vertex of $T_i$ that has degree 3 in $T$.
    \end{claim}
    
    \begin{proof}
        We will prove the claim for $T_1$; it will follow identically for the other $T_i$. We first show that $T_1$ contains at most one leaf $x$ with $d(x,y) \ge 3$. Let $C'$ be the collection of colours incident to $v_1$ and let $C$ be a collection of $m+2$ colours such that $C' \subseteq C$ (since $y$ is the only vertex of $\{y\} \cup \bigcup_{i=1}^{k}V(T_i)$ with degree at least $m+3$). Let $\phi'$ be obtained from $\phi$ by properly recolouring $T_x$ using colours from $C$, for every $x \in T_{1}$ with $d(y,x) = 3$. This is possible because for each such $x$, by Lemma~\ref{lem:broom4_containing_tree} $\phi(f(x)x) = \phi(y v_1)$ and every vertex in $T_x$ has degree at most $m+2$ in $T$.

        $(T,\phi')$ contains no rainbow $B_{4,m}$. To see this, observe that such a $B_{4,m}$ would have to contain $P_{xy}$ for some vertex $x \in T_1$ with $d(x,y) \ge 3$. However, since $d(y) \ge m+3$, by Lemma~\ref{lem:broom4_containing_tree}, each such path contains a repeated colour, namely $\phi(yv_1)=\phi(xf(x))$. 

    Suppose $T$ contains leaves $u,w$ in $T_1$ such that $d(y,u)$ and $d(y,w)$ are both at least 3. Then the addition of $uw$ in a colour from $C$ not incident to $u$ or $w$ cannot create a rainbow $B_{4,m}$ (such a $B_{4,m}$ would have to contain a path that passes through $y$ and hence contain a repeated colour, as in the previous paragraph). So there is at most one such vertex.

    Now suppose $v_1$ has two distinct leaf neighbours $u,w$. Then the addition of $uw$ in the colour $\phi(yv_1)$ preserves the proper colouring, and any copy of $B_{4,m}$ containing $uw$ either receives only colours in $C$ or contains the repeated colour $\phi(yv_1)$. So $v_1$ has at most one leaf neighbour.

    Using Lemma~\ref{lem:broom4_containing_tree}, we obtain that there can be no degree 3 vertex $x$ in $T_1$ with $d(x,y) \ge 2$. The claim follows.
    \end{proof}
    
    \begin{claim} \label{cl:notnonpathbroom}
        If $m>1$, then $T$ is not properly rainbow $B_{4,m}$-saturated.
    \end{claim}
    \begin{proof}
        Suppose $m>1$.
        For each $1\le i \le k$, let $u_i$ be the vertex in $T_i$ of maximal distance from $y$ ($u_i$ is unique by Claim~\ref{cl:fewleaves2broom}).
        Let $x_i$ be the neighbour of $v_i$ in $P_{v_iu_i}$ (if it exists), and let $z_i$ be the neighbour of $v_i$ not in $P_{v_iu_i}$ (if it exists; $z_i$ is a leaf by Claim~\ref{cl:fewleaves2broom}).
        Let $c_0:=\phi(yy')$, and for each $i \in [k]$, let $c_i:=\phi(yv_i)$.
        
        For each $i \in [k]$, if $x_i$ exists, properly recolour the path $P_{x_iu_i}$ in alternating colours $c_i,\phi(v_ix_i)$, and call the resulting colouring $\phi'$.
        Note that any copy of $B_{4,m}$ that contains any recoloured edge in some $P_{x_iu_i}$ contains at least three edges in the $2$-coloured path $P_{yu_i}$ and is thus not rainbow.
        So $(T,\phi')$ is rainbow $B_{4,m}$-free.
        
        \textbf{Case 1:} Suppose that there exist distinct $i,j \in [k]$ such that $d(u_i,y)\ge 3$; without loss of generality, $i=1$ and $j=2$.
        Then, since $d(v_1),d(v_2)\le 3$ and $d(v) \le 2$ for all $v \in (V(T_1)\setminus\{v_1\})\cup(V(T_2)\setminus\{v_2\})$, adding $u_1u_2$ to $T$ only creates a copy $B$ of $B_{4,m}$ if $m=2$ and every edge incident with either $v_1$ or $v_2$ is in $B$.
        But then $B$ contains three edges in one of the $2$-coloured paths $P_{yu_1}$ or $P_{yu_2}$ and is thus not rainbow.
        Thus, $T$ is not properly rainbow $B_{4,m}$-saturated.

        \textbf{Case 2:} Suppose that there exists $i \in [k]$ such that $d(u_j,y)\le 2$ for all $j\neq i$.
        Since $k\ge m+3>2$, there exist distinct $j,j' \in [k]\setminus\{i\}$ such that $d(u_j,y)=d(u_{j'},y) \in \{1,2\}$; without loss of generality, $j=1$ and $j'=2$.
        Note that by Claim~\ref{cl:fewleaves2broom}, $v_1$ and $v_2$ each cannot have two leaf neighbours, so $d(v_1)=d(v_2)\le 2$.
        
        If $v_1$ and $v_2$ are leaves, then by adding $u_1u_2=v_1v_2$ to $T$ in colour $c_0$, any copy of $B_{4,2}$ containing $v_1v_2$ (and thus any copy of $B_{4,m}$) either contains the edge $yy'$ or contains every edge incident with $v_i$ and thus has the repeated colour $c_0$.        
        Otherwise, adding the edge $yu_1=yx_1$ in any available colour, any copy of $B_{4,m}$ containing both $yx_1$ and $x_1v_1$ either contains the edge $yy'$ or contains every edge incident with $v_i$ and thus has the repeated colour $c_0$.
        If there is a copy $B$ of $B_{4,m}$ containing $yx_1$ and not $x_1v_1$, then since edges in $E(B)\setminus\{v_1x_1\}$ receive $m+2$ colours and $d_T(y)\ge m+3$, it is possible to replace $yx_1$ with an edge incident with $y$ in $T$ to form a rainbow copy of $B_{4,m}$ contained in $(T,\phi')$.
        But $(T,\phi')$ is rainbow $B_{4,m}$-free, so this is a contradiction.

        In all cases, $T$ is not properly rainbow $B_{4,m}$-saturated, as desired.
    \end{proof}
    By Claim~\ref{cl:notnonpathbroom}, we can now assume that $m=1$, i.e., $B_{4,m}=P_5$.

    \begin{claim}\label{cl:3colsbroom}
    Let $u,w,w',x,x',y,y'$ be vertices with edges $uw, uw', wx, wx', w'y w'y'$. Then the colours incident to $w$ are the same as the colours incident to $w'$.
    \end{claim}

    \begin{proof}
        Let $C = \{a,b,c\}$ and suppose $\phi(uw) = a, \phi(wx) = b, \phi(wx') = c$. If $\phi(uw') \notin C$, then to avoid a rainbow $P_5$, both $\phi(w'y)$ and $\phi(w'y')$ must be $a$, a contradiction as $\phi$ is a proper edge colouring. So, without loss of generality $\phi(uw') = b$. But now, if either $\phi(w'y)$ or $\phi(w'y')$ are not in $C$, then we have a rainbow-$P_5$. This completes the claim.
    \end{proof}
 
    \begin{claim}
        No vertex $v_1,\ldots,v_k$ has degree 3.
    \end{claim}
    \begin{proof}
        Suppose, without loss of generality, both $v_1$ and $v_2$ have degree 3. By Claim~\ref{cl:3colsbroom}, there is a set $C$ of colours, with $|C|=3$, such that every edge incident to $v_1$ and $v_2$ has a colour in $C$. By Claim~\ref{cl:fewleaves2broom}, $v_1$ and $v_2$ each have exactly one leaf neighbour. So there are vertices $x_1,z_1 \in T_1$ and $x_2,z_2 \in T_2$ such that $d(y,x_1) = d(y,x_2)=2, d(y,z_1) = d(y,z_2) = 3$, and both $\phi(x_1z_1)=\phi(yv_1) \in C$ and $\phi(x_2z_2) = \phi(yv_2) \in C$ (using Lemma~\ref{lem:broom4_containing_tree} to get the equality).
        Let $\phi'$ be obtained from $\phi$ by recolouring $T_{z_1}$ and $T_{z_2}$ using colours from $C$. Observe that $(T,\phi')$ contains no rainbow $P_5$. But now, adding an edge between leaves of $T_{z_1}-z_1$ and $T_{z_2}-z_2$ in an available colour from $C$ will not create a rainbow copy of $P_5$ in $(T,\phi')$, a contradiction. 

        Now suppose $v_1$ is the unique vertex of degree 3 in $\{v_1,\ldots,v_k\}$ and let $N(v_1) = \{x,y,z\}$, where $x$ is a leaf. Suppose $\phi(v_1,y) = a, \phi(v_1,x) = b$, and $\phi(v_1,z) = c$.  For a leaf $t$, let $e_t$ denote the edge of $T$ incident to $t$. We have $T_z = P_{zu}$ for some $u \in T$ and we may assume, by recolouring if necessary, that the colours on $P_{yu}$ alternate between $a$ and $c$ and $\phi(e_u) = a$ (else $\phi(e_u) = c$ and adding $xu$ in colour $a$ will not create a rainbow $P_5$). 

        First consider the case that $\phi(yv_i) = d \in \{b,c\}$ for some $i > 1$.  By recolouring, we may assume that $T_i$ is a path with colours alternating between $a$ and $d$. (Note that if such a recolouring created a rainbow $P_5$, then $\phi$ would have already contained one using vertices from $\{v_1,x,z\}$.) Let $w$ be the leaf in $T_i$. If $\phi(e_w) = d$, then add $xw$ in colour $a$. This cannot create a rainbow $P_5$ else $(T,\phi)$ had one already. Otherwise, $\phi(e_w) = a$ and adding $wu$ in colour $d$ yields no rainbow $P_5$. 

        But now we may assume that the only edge incident to $y$ with colour in $\{b,c\}$ is $yy'$. But now we may recolour $yv_2$ using a colour $d$ from $\{b,c\}\setminus \{\phi(yy')\}$ and recolour $T_2$ in colours alternating between $d$ and $a$. As before, such a recolouring cannot create a rainbow $P_5$, else we would have had one already. So we may assume we are in the previous case and we are done.        
    \end{proof}   

    So now every vertex in $\bigcup_{i=1}^k T_i$ has degree at most 2 in $T$. By recolouring if necessary, we may assume that each $T_i$ is a path with alternating colours $\phi(yy') = d$ and $\phi(yv_i)$.

    For each $i$, let $x_i$ be the leaf in $T_i$. If $\phi(e_{x_i}),\phi(e_{x_j}) \not= d$ for $i\not =j $, then adding $x_ix_j$ coloured $d$ creates no rainbow $P_5$. But now, as $i \ge 3$, there exists $i $ such that $\phi(e_{x_i}) = d$ and adding $x_iy$ in any colour will yield no rainbow $P_5$. This completes the proof.    
\end{proof}

We will now prove the upper bound on $\prsat(n,B_{4,m})$ from Theorem~\ref{thm:broom4}.

\begin{theorem} \label{thm:broom4_upper}
For all $m \in \bb{N}$ and $n \ge 3(m+2)$, 
\[\prsat(n,B_{4,m}) \le 3(m+2)\left\lfloor \frac{n}{3(m+2)} \right\rfloor + \binom{n-3(m+2)\left\lfloor \frac{n}{3(m+2)} \right\rfloor}{2}.\]
In particular, for fixed $m$, $\prsat(n,B_{4,m}) \le n+O(1)$ and $\prsat(n,B_{4,m}) \le n-1$ for infinitely many $n$.
\end{theorem}
\begin{proof}
\begin{figure}
    \centering
    \includegraphics{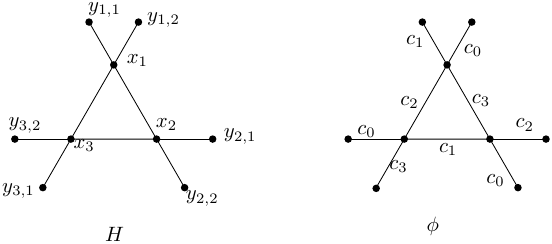}
    \caption{Graph $H$ from Theorem \ref{thm:broom4_upper} in the case $m=1$ (left) and an arbitrary rainbow $B_{4,1}(=P_5)$-free colouring $\phi$ of $H$ (right).}
    \label{fig:p5_fig4}
\end{figure}
Let $H$ be the $3(m+2)$-vertex graph obtained from the triangle on vertices in $X=\{x_1,x_2,x_3\}$ by appending $m+1$ pendant vertices $y_{i,0},y_{i,1},\ldots,y_{i,m}$ to each $x_i$.
For each $i \in [3]$, let $Y_i=\{y_{i,0},\ldots,y_{i,m}\}$, and let $Y=\bigcup_{i=1}^3 Y_i$.

Note that to show that a broom $B\cong B_{4,m}$ with path $v_1\bc v_2\bc v_3\bc v_4$ and vertices $u_1,\ldots,u_m \in N_B(v_4)\setminus\{v_3\}$ is rainbow, it suffices to show $\phi(v_1v_2) \neq \phi(v_3v_4)$ and 
\[ \phi(\{v_1v_2,v_2v_3\})\cap \phi(v_4\{u_1,\ldots,u_m\})=\emptyset,\]
since the distinctness of the other pairs of edges follows from edge adjacency.
Let $\phi:E(H)\to [m+3]$ be an edge-colouring of $H$ defined as follows: for each $i \in \{1,2,3\}$, letting $\{j,k\} = \{1,2,3\}\setminus \{i\}$, set $\phi(x_iy_{i,0}) = \phi(x_jx_k) = i$, and for $\ell \in [m]$, set $x_iy_{i,\ell} = \ell+3$.
\begin{claim} \label{clm:broom4_upper_no_rainbow}
$\phi$ is a rainbow $B_{4,m}$-free proper colouring of $H$.
\end{claim}
\begin{proof}[Proof of Claim]

For all $i \in [3]$, $x_i$ is incident to exactly $m+3$ edges and is incident to edges of every colour in $[m+3]$, and all vertices besides $x_1,x_2,x_3$ are leaves, so $\phi$ is proper.
Take any copy $B$ of $B_{4,m}$ in $H$.
$B$ must contain a length-$2$ path in $H[X]$, say $x_1,x_2,x_3$, else its longest path has length at most $3$.
Thus, without loss of generality, $B$ contains $m$ vertices in $Y_1$ and $1$ vertex $y_3\in Y_3$.
If $y_{1,0} \in V(B)$, then $B$ has two edges $x_1y_{1,0},x_2x_3$ of the colour $c_1$, so $B$ is not rainbow; similarly, $B$ is not rainbow if $y_{3,0} \in V(B)$.
Otherwise, edges from $x_1$ to $V(B)\cap Y_1$ receive every colour in $[4,m+3]$, and $\phi(x_3y_3) \in [4,m+3]$, so $\phi(x_3y_3)$ is a repeated colour in $B$.
Therefore, $\phi$ is a rainbow $B_{4,m}$-free proper colouring of $H$.
\end{proof}
Let $\phi'$ be any rainbow $B_{4,m}$-free proper colouring of $H$.
Let $c_1 = \phi'(x_2x_3)$, $c_2 = \phi'(x_1x_3)$, and $c_3 = \phi'(x_2x_3)$.
Let $C$ be a $m$-subset of $\phi'(x_1Y_1)\setminus\{c_1\}$, which exists since $|Y_1|=m+1$ and edges incident with $x_1$ receive distinct colours.
Note that $c_2,c_3 \notin C$.

\begin{claim} \label{clm:broom4_upper_characterization}
For all $i \in [3]$, $\phi'(x_iY_i)=\{c_i\}\cup C$.
\end{claim}
\begin{proof}[Proof of Claim]
Let $(i,j,k)$ be any ordering of $\{1,2,3\}$.
Then $\phi'(x_iY_i) \subseteq \phi'(x_jY_j) \cup \{c_i\}$ and $c_i \in \phi'(x_iY_i)$, since otherwise, letting $y_i \in Y_i$ such that $\phi'(x_iy_i)\notin \phi'(x_jY_j)\cup\{c_i\}$ and $Y_j'\subseteq Y_j$ be an $m$-subset such that $c_j\notin \phi'(x_jY_j')$, we have that edges in the path $y_i,x_i,x_k,x_j$ and in $x_jY_j'$ form a rainbow copy of $B_{4,m}$ in colours $\{\phi'(y_ix_i),c_j,c_i\}\cup \phi'(x_jY_j')$.
Similarly, $c_i \in \phi'(x_iY_i)$, since otherwise, letting $Y_j'\subseteq Y_j$ be an $m$-subset such that $c_j\notin \phi'(x_jY_j')$ and letting $y_i \in Y_i$ such that $\phi'(x_iy_i)\notin \phi'(x_jY_j')$, we have that edges $y_i,x_i,x_k,x_j$ and $x_jY_j'$ form a rainbow copy of $B_{4,m}$.
So $\phi'(x_1Y_1) = \{c_1\}\cup C$.
Now, $c_1\notin \phi'(x_2Y_2\cup x_3Y_3)$, since edges in $x_2Y_2 \cup x_3Y_3$ are adjacent to $x_2x_3$, which has colour $c_1$.
Thus, $\phi'(x_2Y_2) = \{c_2\} \cup (\phi'(x_1Y_1)\setminus\{c_1\}) = \{c_2\}\cup C$, and similarly, $\phi'(x_3Y_3) = \{c_3\}\cup C$, as desired.
\end{proof}
By Claim~\ref{clm:broom4_upper_characterization} and by symmetry, we can assume that any rainbow $B_{4,m}$-free proper colouring of $H$ is equivalent to $\phi$ up to relabelling vertices and colours.

\begin{claim} \label{clm:broom4_upper_has_rainbow}
Let $G$ be a graph that contains $H$ as a subgraph.
For all $uv \in E(\overline{G})$ such that $u \in V(H)$, $G+uv$ contains a rainbow copy of $P_5$ in any proper colouring.
\end{claim}
\begin{proof}[Proof of Claim]
Suppose for contradiction that there exists a rainbow $B_{4,m}$-free proper colouring $\phi'$ of $G+uv$.
Since $\phi'|_{E(H)}$ is a rainbow $B_{4,m}$-free proper colouring of $H$, without loss of generality, $\phi'$ is an extension of $\phi$ to $E(G+uv)$.

\textbf{Case 1:} $u,v \notin X$.
Without loss of generality, $u \in Y_1$, and either $v \in Y_2$ or $v\notin V(H)$.
If $\phi(ux_1) \in [4,m+1]$, then letting $i \in \{2,3\}$ such that $\phi(x_1x_i) \neq \phi(uv)$ and letting $N\subseteq N_H(x_i)\setminus\{x_1,v\}$ be an $m$-set such that $\phi(vu),\phi(ux_1)\notin \phi(x_iN)$, we have that edges in the path $v,u,x_1,x_i$ and in $x_iN$ form a rainbow copy of $B_{4,m}$.
Otherwise, we can assume that $\phi(ux_1),\phi(vx_2)\notin [4,m+1]$ (or $v\notin V(H)$), so letting $i \in \{2,3\}$ such that either $\phi(uv)=i$ or $\phi(uv)\notin [3]$ and $v\notin Y_i'$, and letting $Y_i'\subseteq Y_i\setminus\{v\}$ be an $m$-set such that $\phi(uv)\notin \phi(x_iY_i')$, we have that edges in the path $v,u,x_1,x_i$ and in $x_iY_i'$ form a rainbow copy of $B_{4,m}$.

\textbf{Case 2:} $\{u,v\} \cap X\neq \emptyset$.
Without loss of generality, $u=x_1$ and $\{v\}\cap V(H) \subseteq Y_2$ (we cannot have $v\in \{x_2,x_3\}$, else $uv \in E(H)$).
Also, $u$ is incident with edges of every colour in $[m+3]$ in $H$, so $\phi(uv) \notin [m+3]$.
So, letting $Y_3'= Y_3\setminus \{y_{3,1}\}$ (so that $3 \notin \phi(x_3Y_3')$), edges in the path $v,u,x_2,x_3$ and in $x_3Y_3'$ form a rainbow copy of $B_{4,m}$ in colours $\{\phi(uv),3,1\}\cup [4,m+3]$.

In both cases, there exists a rainbow copy of $B_{4,m}$, contradicting $\phi'$ being a rainbow $B_{4,m}$-free colouring of $G$.
Therefore, $G+uv$ contains a rainbow copy of $B_{4,m}$ in any proper colouring, proving Claim \ref{clm:broom4_upper_has_rainbow}.
\end{proof}

Let $G = \ff{n}{3(m+2)} H \cup F$, where $F \in \Prsat(n-3(m+2)\left\lfloor \frac{n}{3(m+2)} \right\rfloor,B_{4,m})$.
By Claim \ref{clm:broom4_upper_has_rainbow}, each component of $G$ has a rainbow $B_{4,m}$-free proper colouring, so $G$ has a rainbow $B_{4,m}$-free proper colouring.
For all $e \in E(\overline{G})$, either $e$ has an endpoint in a copy of $H$ or $e \in E(\overline{F})$, so $G+e$ contains a rainbow copy of $B_{4,m}$ in any proper colouring by Claim \ref{clm:broom4_upper_has_rainbow} and the definition of $F$.
Therefore, $G$ is properly rainbow $B_{4,m}$-saturated, so 
\begin{equation} \label{eqn:broom4_upper}
    \prsat(n,B_{4,m}) \le |E(G)| \le 3(m+2)\left\lfloor \frac{n}{3(m+2)} \right\rfloor + \binom{n-3(m+2)\left\lfloor \frac{n}{3(m+2)} \right\rfloor}{2}.
\end{equation}
For the `in particular' statement, note that $\binom{n-3(m+2)\left\lfloor \frac{n}{3(m+2)} \right\rfloor}{2} \le \binom{3(m+2)-1}{2}= O(1)$, and when $n\equiv 1$ or $2 \pmod{3(m+2)}$, (\ref{eqn:broom4_upper}) yields $\prsat(n,B_{4,m})\le n-1$.
This completes the proof of Theorem~\ref{thm:broom4_upper}.
\end{proof}

\section{Upper bound for caterpillars}\label{sec:cat}
In this section, we prove Theorem~\ref{thm:caterpillar}, restated here for convenience.
\caterpillar*

\begin{figure}
    \centering
    \includegraphics{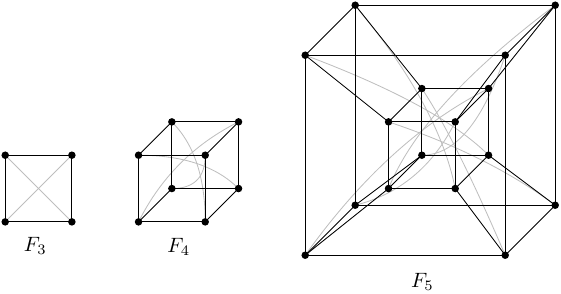}
    \caption{Folded cubes}
    \label{fig:folded_cube}
\end{figure}
For each $\ell \ge 2$, let $F_{\ell+1}$ be the graph with $V(F_{\ell+1})=\bb{F}_2^{\ell}$ and $xy \in E(F_{\ell+1})$ if and only if $y-x \in A_{\ell+1} :=\{\vec{1},e_1,e_2,\ldots,e_{\ell}\}$, where $\vec{1}$ is the all-ones vector and $e_i$ is the vector with $1$ in the $i$th entry and $0$ in all other entries.
We refer to the graphs $F_{\ell+1}$ as \emph{folded cubes}.
Note that $A_{\ell+1}$ is in general position, i.e., when considering $\bb{F}_2^\ell$ as a vector space over $\bb{F}_2$, any subset $S\subseteq A_{\ell+1}$ with $|S|\le \ell$ is linearly independent.
In other words, for all $\emptyset \neq S \subsetneq A_{\ell+1}$, $\sum_{a \in S}a\neq 0$.
However, $\sum_{a \in A_{\ell+1}}a = 0$.
Also, for all $a \in A_{\ell+1}$, the spanning subgraph of $F_{\ell+1}$ with edge set $\{xy\in E(F_{\ell+1}):y-x\neq a\}$ is isomorphic to the hypercube $Q_\ell$.

In this section, for a colouring $\phi$ of $F_{\ell+1}$ and $x,y \in V(F_{\ell+1})$, we will write $\phi(x,y)$ in place of $\phi(xy)$ to avoid confusion with multiplication.

Folded cubes have been studied in the context of the rainbow Tur\'an problem (see e.g. \cite{johnston2020lower}, \cite{halfpap2021rainbowcycles}).
The following result gives the best known general lower bound on the rainbow Tur\'an number for paths.
\begin{theorem}[Johnston, Rombach~\cite{johnston2020lower}] \label{thm:path_hypercube}
Let $\ell \ge 3$.
Then there exists a rainbow $P_\ell$-free proper colouring of $F_{\ell-1}$, namely the colouring $\phi:E(F_{\ell-1})\to A_{\ell-1}$ defined by $\phi(x,y)=y-x$.
\end{theorem}

In order to prove Theorem~\ref{thm:caterpillar}, we have the following technical lemma.
For a walk $w=x_1,\ldots,x_k$ in a graph $G$ with colouring $\phi$, say $w$ is a \emph{rainbow trail} if for all $1\le i<j<k$, $\phi(x_i,x_{i+1})\neq\phi(x_j,x_{j+1})$.
\begin{lemma} \label{lem:path_con_holds}
Let $\ell \ge 4$.
Let $\phi$ be a rainbow $P_\ell$-free proper colouring of $F_{\ell-1}$.
Then the following hold.
\begin{enumerate}
    \item $\phi$ is an $(\ell-1)$-colouring.
    \item For any $x \in V(F_{\ell-1})$ and any $c \in \phi(E(F_{\ell-1}))$, $\phi$ contains a rainbow $(\ell-1)$-vertex path with endpoint $x$ that avoids the colour $c$.
    \item For any $x,y \in V(F_{\ell-1})$ with $x\neq y$, $\phi$ contains a rainbow $(\ell-1)$-vertex path with endpoint $x$ that avoids the vertex $y$.
    \item $F_{\ell-1}$ is properly rainbow $P_\ell$-saturated.
\end{enumerate}
\end{lemma}
\begin{proof}
Let $\ell\ge 4$, and let $\phi$ be a rainbow $P_\ell$-free proper colouring of $F=F_{\ell-1}$.
We will show that in every subgraph of $F$ with large minimum degree, every $P_3$ can be extended to a rainbow $P_{\ell-1}$, which in turn extends to a rainbow $P_\ell$ in $F$ unless $\phi$ has the desired properties.
\begin{claim} \label{clm:Pk_upper_path}
Let $F'$ be a subgraph of $F$ with $\delta:=\delta(F')\ge \ell-2$.
Let $x_0,x_1,x_2$ be a path in $F'$.
For each $j \in [\ell-2]$, let $m_j:= \min\{j-1,\delta-j+1\}$.
Then there exist $a_1,\ldots,a_{\ell-2} \in A_{\ell-1}$ and $x_3,\ldots,x_{\ell-2} \in V(F)$ such that the following hold for all $j \in [\ell-2]$:
\begin{enumerate}[label=(\roman*)]
    \item $x_j=x_{j-1}+a_j$.
    \item $a_j \notin \{a_{j-m_j},\ldots,a_{j-1}\}$.
    \item $p_j:= x_0,x_1,\ldots,x_j$ is a rainbow trail.
    \item $p_j=x_0,x_1,\ldots,x_j$ is a path.
\end{enumerate}
In particular, $p_{\ell-2}=x_0,x_1,\ldots,x_{\ell-2}$ is a rainbow path.
\end{claim}
\begin{proof}[Proof of Claim]
We will show (i)--(iii) by induction on $j$.
By setting $a_1=x_1-x_0$ and $a_2=x_2-x_1$, we are done when $j \le 2$ ($3$-vertex paths are trivially rainbow), so suppose $j \ge 3$.
By the induction hypothesis, let $a_1,\ldots,a_{j-1}$ and $x_3,\ldots,x_{j-1}$ be such that (i)--(iii) hold for all $j'\le j-1$.
Now $|\{a_{j-m_j},\ldots,a_{j-1}\}|= m_j \le \delta-j+1$ and $p_{j-2}$ receives $j-2$ colours, so because $d_{F'}(x_{j-1})-(\delta-j+1)-(j-2) \ge 1$, there exists $a_j \in A_{\ell-1}\setminus\{a_{j-m_j},\ldots,a_{j-1}\}$ such that by setting $x_j=x_{j-1}+a_j$, we have that $p_j=x_0,x_1,\ldots,x_j$ is rainbow.
Thus, by induction, (i)--(iii) hold for all $j$, so there exist $a_1,\ldots,a_{\ell-2}$ such that $p_{\ell-2}$ is rainbow and $a_i\notin \{a_{i-m_i},\ldots,a_{i-1}\}$ for all $i \in [\ell-2]$.

We will now prove (iv).
Suppose for contradiction that for some $0\le j< j' \le \ell-2$, $x_j = x_{j'}$.
Then 
\begin{equation} \label{eqn:cat_upper}
    r:=\sum_{i=j+1}^{j'}a_i =x_{j'}-x_j =0.
\end{equation}
Since $A_{\ell-1}$ is in general position, the set $\{a_{j+1},\ldots,a_{j'}\}$ is linearly independent (ignoring repeated elements), so (\ref{eqn:cat_upper}) is only possible if every element of $A_{\ell-1}$ appears an even number of times in the sum $r$.
Thus, $j'-j$ is even, and $[j+1,j']$ can be partitioned into $\frac{j'-j}{2}$ pairs $\{i,i'\}$ such that $a_i=a_{i'}$.
By (ii), for a fixed $i \in [j+1,j']$, there can exist $i' \in [j+1,i-1]$ such that $a_i=a_{i'}$ only if $i'<i-m_i$, i.e., $i'\le 2i-\delta-2$.
So
\begin{align*}
    |\{i \in [j+1,j']:\exists i' \in [j+1,i-1],\, a_i=a_{i'}\}| &\le |\{i \in [j+1,j']:j+1 \le 2i-\delta-2 \}| \\
    &= \left|\left\{i \in [j+1,j']:i\ge \frac{j+\delta+3}{2}\right\}\right| 
    \end{align*}
This is at most 
$$    j'-\frac{j+\delta+3}{2}+1 \le j'-\frac{j+(\ell-2)+3}{2}+1 \le j'-\frac{j+j'+3}{2}+1 = \frac{j'-j-1}{2}.$$
Thus, $[j+1,j']$ can be partitioned into at most $\frac{j'-j-1}{2}<\frac{j'-j}{2}$ pairs $\{i,i'\}$ such that $a_i=a_{i'}$, a contradiction.
So $x_0,\ldots,x_{\ell-2}$ are pairwise disjoint, proving (iv).
\end{proof}
We will now prove conditions (1)--(4) from the lemma statement.

(1)
Suppose $\phi$ is not an $(\ell-1)$-colouring.
For each $x \in V(F)$, define $\phi(x)$ to be the set of colours of edges incident with $x$.
There exists an edge $x_0x_1 \in E(F)$ such that $\phi(x_0)\neq \phi(x_1)$: otherwise, $\phi(u)=\phi(v)$ for all $u,v \in V(F)$ since $F$ is connected, and so $\phi$ is an $(\ell-1)$-colouring.
In particular, there exists $x_2 \in N_F(x_1)$ such that $\phi(x_1,x_2) \notin \phi(x_0)$ (and thus $x_2\neq x_0$).
Let $a_1,\ldots,a_{\ell-2}$ and $x_3,\ldots,x_{\ell-2}$ be as in Claim~\ref{clm:Pk_upper_path} for $F'=F$, and note that $p_{\ell-2}=x_0,x_1,\ldots,x_{\ell-2}$ is a rainbow path.

Let $\ca{B}$ be an $(\ell-2)$-subset of $A_{\ell-1}$ that contains $\{a_1,\ldots,a_{\ell-2}\}$.
Then since $A_{\ell-1}$ is in general position, $\ca{B}$ is a basis for $\bb{F}_2^{\ell-2}$.
Let $s$ be the unique element of $A_{\ell-1}\setminus \ca{B}$.
Since $x_2,x_3,\ldots,x_{\ell-2}$ receives $\ell-4$ colours and $|\ca{B}\setminus\{a_1\}|-(\ell-4)\ge 1$, there exists $a_0 \in \ca{B}\setminus\{a_1\}$ such that, setting $x_{-1} := x_0+a_0$, we have that $p:= x_{-1},x_0,\ldots,x_{\ell-2}$ is a rainbow trail (note that $\phi(x_{-1},x_0)\neq \phi(x_1,x_2)$ since $\phi(x_1,x_2) \notin \phi(x_0)$).
Thus, it now suffices to show that $p$ is a path, i.e., $x_{-1} \notin \{x_0,\ldots,x_{\ell-2}\}$.

For each $j \in [-1,\ell-2]$, let        
\[x_j -x_0 = \sum_{b\in \ca{B}}\alpha_{j,b}b\]
for some $\{\alpha_{j,b}\}_{b \in \ca{B}}$ in $\bb{F}_2$, and let 
\[f(j):= |\{b\in \ca{B}:\alpha_{j,b}\neq 0\}|.\]
$f(j)$ can be interpreted as the distance between $x_j$ and $x_0$ in the copy of $Q_{\ell-2}$ in $F$ generated by the directions in $\ca{B}$.
Note that $f(0) =0$, and for each $j \ge 1$, $f(j) \in \{f(j-1)\pm 1\}$.
For all $1\le j \le \lceil \ell/2\rceil$, letting $m_j$ be as in Claim~\ref{clm:Pk_upper_path} with $F'=F$, we have
\[ j-m_j \le j-(\delta(F)-j+1) = 2j-\ell \le 2\left\lceil\frac{\ell}{2}\right\rceil -\ell \le 1. \]
Thus, $a_j \notin \{a_{j-m_j},\dots,a_{j-1}\}=\{a_1,\ldots,a_{j-1}\}$ by Claim~\ref{clm:Pk_upper_path}(ii), so $f(j) = f(j-1)+1$.
Thus, by iteratively applying this equality, $f(j) = j$ for all $1\le j \le \lceil \ell/2 \rceil$.
For each $\lceil \ell/2 \rceil < j \le \ell-2$, $f(j) \ge f(j-1)-1$, so by iteratively applying this bound, we obtain 
\[f(j) \ge f\left(\left\lceil \frac{\ell}{2} \right\rceil\right)-\left(j-\left\lceil\frac{\ell}{2}\right\rceil\right) = 2\left\lceil\frac{\ell}{2}\right\rceil-j \ge 2\left\lceil\frac{\ell}{2}\right\rceil-(\ell-2) >1.\]
Now $a_0 \in \ca{B}$ by definition, and $x_{-1}=x_0+a_0$, so $f(-1) = 1$.
The only element $j \in [0,\ell-2]$ with $f(j) = 1$ is $j=1$, so in particular, $x_{-1}\neq x_j$ for all $j \in [0,\ell-2]\setminus\{1\}$.
But $a_0\neq a_1$ by definition, so $x_{-1}\neq x_1$.
Therefore, $x_{-1}\notin \{x_0,\ldots,x_{\ell-2}\}$, so $p$ is a rainbow path on $\ell$ vertices.
This is a contradiction, as $\phi$ is rainbow $P_\ell$-free.
Therefore, $\phi$ is a $(\ell-1)$-colouring.

(2)
Let $x=x_0 \in V(F)$ and $c \in \phi(E(F))$.
Let $E_c$ be the colour class of $c$ in $F$, and let $F':= F-E_c$.
Note that $\delta(F')\ge \delta(F)-1=\ell-2$.
Let $x_0,x_1,x_2$ be a path in $F'$ with endpoint $x_0$, which exists since $\delta(F')\ge \ell-2\ge 2$.
By Claim~\ref{clm:Pk_upper_path}, $p_{\ell-2}=x_0,x_1,\ldots,x_{\ell-2}$ is a rainbow $(\ell-1)$-vertex path in $F'$ with endpoint $x$, as desired.

(3)
Let $x=x_0 \in V(F)$ and $y \in V(F)$.
Let $F':= F-y$.
Note that $\delta(F') = \delta(F)-1=\ell-2$.
Let $x_0,x_1,x_2$ be a path in $F'$ with endpoint $x_0$, which exists since $\delta(F')= \ell-2\ge 2$.
By Claim~\ref{clm:Pk_upper_path}, $p_{\ell-2}=x_0,x_1,\ldots,x_{\ell-2}$ is a rainbow $(\ell-1)$-vertex path in $F'$ with endpoint $x$, as desired.

(4)
$F$ has a rainbow $P_\ell$-free proper colouring by Theorem~\ref{thm:path_hypercube}.
Suppose for contradiction that there exist $xy \in E(\overline{F})$ and a rainbow $P_\ell$-free proper colouring $\phi$ of $F+xy$.
Let $F':=F-y = (F'+xy)-y$.
Note that $\delta(F')=\delta(F)-1=\ell-2$.
Let $x_0,x_1,x_2$ be a path in $F'$ with endpoint $x_0=x$, which exists since $\delta(F')\ge \ell-2 \ge 2$.
By Claim~\ref{clm:Pk_upper_path}, $p_{\ell-2}=x_0,x_1,\ldots,x_{\ell-2}$ is a rainbow $(\ell-1)$-vertex path in $F'$.

Now $\phi|_{E(F)}$ is a rainbow $P_\ell$-free proper colouring, so by (1), $\phi|_{E(F)}$ is an $(\ell-1)$-colouring.
Thus, $\phi(x,y) \notin \phi(E(F))$, since in $F$, $x$ is already incident to edges of every colour in $\phi(E(F))$.
So $\phi(x,y)$ is not present in the colours of $p_{\ell-2}$, which are contained in $\phi(E(F'))\subseteq \phi(E(F))$.
Since $p_{\ell-2}$ is in $F'=F-y$, we have $y \notin \{x_0,\ldots,x_{\ell-2}\}$, so $y,p_{\ell-2}=y,x_0,x_1,\ldots,x_{\ell-2}$ is a rainbow $\ell$-vertex path in $F+xy$.
This is a contradiction.
Therefore, $F$ is properly rainbow $P_\ell$-saturated, proving (4).

This completes the proof of Lemma~\ref{lem:path_con_holds}.
\end{proof}

We are now ready to prove Theorem~\ref{thm:caterpillar}.
\begin{proof}[Proof of Theorem~\ref{thm:caterpillar}]
Let $F=F_{\ell-1}$, and let $G$ be an $n$-vertex graph obtained from $F$ by adding at least $k$ pendant edges to each vertex of $F$.
Since any path in $G$ has at most two leaves, it must contain at most $2$ edges not in $F$.
Thus, since there exists a rainbow $P_{\ell}$-free proper colouring of $F$ by Theorem~\ref{thm:path_hypercube}, any extension of this colouring to a proper colouring of $G$ is rainbow $P_{\ell+2}$-free and thus rainbow $T_{k,\ell}$-free since $T_{k,\ell} \supset P_{\ell+2}$.

Take any $uv \in E(\overline{G})$, and let $\phi$ be a proper colouring of $G+uv$.
\begin{claim} \label{clm:CP_rainbow_path}
There exists a rainbow path $x_1\bc \ldots\bc x_{\ell+1}$ in $(G+uv,\phi)$ such that $x_1\bc \ldots\bc x_{\ell-1} \in V(F)$.
\end{claim}
\begin{proof}[Proof of Claim]
We will show that $\phi$ contains a rainbow copy of $P_{\ell+1}$.
If there exists a rainbow path $P=x_1,\ldots,x_{\ell}$ in $(F,\phi|_{E(F)})$, then since there are $k\ge \ell$ pendant edges incident with each of $x_1$ and $x_{\ell}$ in $G$, there exists $x_{\ell+1} \in N_G(x_{\ell})$ such that $x_{\ell}x_{\ell+1}$ does not receive a colour in the path $x_1,\ldots,x_{\ell-1}$, so $P,x_{\ell+1}$ is a rainbow $(\ell+1)$-vertex path.
Thus, we can assume that $\phi|_{E(F)}$ is rainbow $P_{\ell}$-free.
In particular, since $F$ is properly rainbow $P_{\ell}$-saturated by Lemma~\ref{lem:path_con_holds}(4), we can assume that $\{u,v\}\not\subseteq V(F)$.

\textbf{Case 1:}
Suppose $u,v \notin V(F)$.
Let $x$ and $y$ be the neighbours of $u$ and $v$, respectively, in $F$.
Then by Lemma~\ref{lem:path_con_holds}(2), $\phi|_{E(F)}$ contains a rainbow path $P=x_1,\ldots,x_{\ell-1}$ with $x_{\ell-1}=x$ that avoids the colour $\phi(vu)$.
By Lemma~\ref{lem:path_con_holds}(1), $\phi|_{E(F)}$ is an $(\ell-1)$-colouring, and $F$ is $(\ell-1)$-regular, so $ux$ does not receive a colour in $\phi(E(F))$.
Therefore, we have that $P,u,v$ is a rainbow $(\ell+2)$-vertex path in $G+uv$ with $V(P) \subseteq V(F)$.

\textbf{Case 2:}
Suppose $v \in V(F)$.
Let $x$ be the neighbour of $u$ in $F$.
Then by Lemma~\ref{lem:path_con_holds}(3), $\phi|_{E(F)}$ contains a rainbow path $P=x_1,\ldots,x_{\ell-1}$ with $x_{\ell-1}=v$ that avoids the vertex $x$.
By Lemma~\ref{lem:path_con_holds}(1), $\phi|_{E(F)}$ is an $(\ell-1)$-colouring, and $F$ is $(\ell-1)$-regular, so neither of $xu,uv$ receives a colour in $\phi(E(F))$.
Therefore, $P,u,x$ is a rainbow $(\ell+1)$-vertex path in $G+uv$ with $V(P)\subseteq V(F)$.

In all cases, there exists a rainbow path $x_1\bc \ldots\bc x_{\ell+1}$ in $(G+uv,\phi)$ such that $x_1\bc \ldots\bc x_{\ell-1} \in V(F)$.
This completes the proof of Claim~\ref{clm:CP_rainbow_path}.
\end{proof}
For each $i \in [\ell-1]$, let $s_i$ be the number of leaf neighbours of $v_i$ in $T_{k,\ell}$.
Starting with the path $P=x_1,\ldots,x_{\ell+1}$ from Claim~\ref{clm:CP_rainbow_path} as our central path, we will add $s_i$ pendant edges incident to each $x_i$ for $i \in [\ell-1]$ to obtain a rainbow copy $T$ of $T_{k,\ell}$.
Suppose that for some $i \in [\ell-1]$, fewer than $s_i$ pendant edges incident to $x_i$ have been added.
Note that $T_{k,\ell}$ has $k-1$ total edges, so at most $k-2$ edges of $T$ have been defined, at least one of which is in $P$ and incident to $x_i$, so there are at most $k-3$ edges incident to $x_i$ whose addition would repeat a colour already present.
Now, $x_i$ has at least $k-2$ neighbours $y$ in $V(G)\setminus (V(F)\cup\{x_{\ell},x_{\ell+1}\})$, so by adding $x_iy$ to $T$ so that $x_iy$ does not repeat a colour already present and repeating this process, we obtain a rainbow copy of $T_{k,\ell}$.

Therefore, $G$ is properly rainbow $T_{k,\ell}$-saturated, so
\[ \prsat(n,T_{k,\ell}) \le |E(G)| = n-2^{\ell-2}+\frac{2^{\ell-2}(\ell-1)}{2} = n+(\ell-3)2^{\ell-3}, \]
as desired.
\end{proof}
\section{Proof of Theorem~\ref{thm:cat_converse}}\label{sec:catcon}
\
The goal of this section is to prove Theorem~\ref{thm:cat_converse}, restated here for convenience.

\catcon*

We require the following standard fact about spanning trees (which can be proved by induction). 
\begin{lemma} \label{lem:spanning_tree_dist}
Let $G$ be a connected graph, and let $x \in V(G)$.
Then there exists a spanning tree $T$ of $G$ such that for all $v \in V(G)$, $d_T(x,v)=d_G(x,v)$.
\end{lemma}

For a graph $H$, a set $S\subseteq V(H)$ and a vertex $v\in V(H)$, let $d_H(v,S):=\min_{u \in S}d_H(v,u)$.

\begin{proof}[Proof of Theorem~\ref{thm:cat_converse}]

Let $n \ge 4$, and let $G \in \Prsat(n,T)$.
Let $D:=\{v\in V(G):d_G(v)\ge 3\}$.
For all $S\subseteq V(G)$, $v \in V(G)$, and $d \in \bb{N}$, let $B_d[S]:=\{u \in V(G):d_G(u,S)\le d\}$ and $B_d[v]:=B_d[\{v\}]$.
\begin{claim} \label{clm:cat_converse_not_far}
$|V(G)\setminus B_r[D]| \le 3$.
\end{claim}
\begin{proof}[Proof of Claim]
Suppose otherwise, and let $A$ be a $4$-subset of $V(G)\setminus B_r[D]$.
Since elements of $A$ each have degree at most $2$, $G[A]$ is not a clique, so let $xy \in E(\overline{G[A]})$.
Since $G \in \Prsat(n,T)$, let $T'$ be a copy of $T$ in $G+xy$ that contains the edge $xy$.
Let $z\in V(T')\setminus\{x,y\}$ with $d_{T'}(z)\ge 3$ such that $d_{T'}(z,\{x,y\})$ is minimal.
Note that $z$ must exist, since otherwise $xy$ is an edge in $T'$ containing all vertices of degree $\ge 3$ in $T'$, which violates the hypothesis on $T$.
Since $x,y \notin B_r[D]$, we have $d_{G+xy}(z,\{x,y\}) = d_G(z,\{x,y\}) > r$, and every vertex in $B_r[\{x,y\}]$ has degree $2$.
So the shortest $z$-$\{x,y\}$ path in $T'$ has length greater than $r$ and has degree-$2$ internal vertices, which contradicts the definition of $r$.
This proves the claim.
\end{proof}

\begin{claim} \label{clm:one_leaf}
Every vertex in $G$ has at most one leaf neighbour.
\end{claim}
\begin{proof}[Proof of Claim]
Suppose there exists $x\in V(G)$ with two leaf neighbours $y$ and $z$.
Let $T'$ be a copy of $T$ in $G+yz$ containing $yz$.
Since $T$ is a tree, it does not contain the cycle $x,y,z,x$, so one of $y$ or $z$ is a leaf in $T'$ and the other has degree $2$ in $T'$.
This contradicts the hypothesis.
Therefore, the claim holds.
\end{proof}

\begin{claim} \label{clm:no_leaf}
At most one vertex of degree $2$ in $G$ has a leaf neighbour.
\end{claim}

\begin{proof}[Proof of Claim]
Suppose for contradiction that there exist $x,y,z,x',y',z' \in V(G)$ such that $N(x) = \{y,z\}$, $N(x')=\{y',z'\}$, $d(y)=d(y') = 1$, $x\neq x'$, $y\neq z$, and $y' \neq z'$.
Let $\phi$ be a rainbow $T$-free proper colouring of $G$.

\begin{figure}
    \centering
    \includegraphics{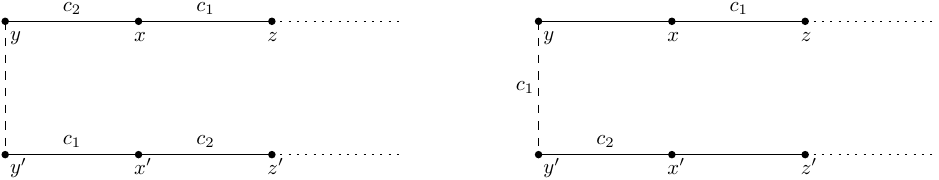}
    \caption{Cases from Claim~\ref{clm:no_leaf}. In both cases, the edge $yy'$ can be added to $(G,\phi)$ in some colour such that no rainbow copy of $T$ is created.}
    \label{fig:cat_converse_no_leaf}
\end{figure}

If $\phi(xz)= \phi(x'y')$ and $\phi(x'z') = \phi(xy)$, extend $\phi$ to a proper colouring of $G+yy'$ by assigning $yy'$ any valid colour.
Then since $T$ has diameter $\ge 4$ and every vertex of degree $2$ in $T$ has no leaf neighbours, any copy of $T$ in $G+yy'$ containing $yy'$ has a repeated colour (see Figure~\ref{fig:cat_converse_no_leaf}, left).
This contradicts $G\in \Prsat(n,T)$.

Otherwise, we can assume $\phi(xz) \neq \phi(x'y')$.
Extend $\phi$ to a proper colouring of $G+yy'$ by setting $\phi(yy'):= \phi(xz)$.
Then any rainbow copy of $T$ in $(G+yy',\phi)$ containing $yy'$ contains a degree-$2$ vertex $y'$ or $y$ with a leaf neighbour $y$ or $x$, respectively, else it repeats the colour $\phi(xz)$ (see Figure~\ref{fig:cat_converse_no_leaf}, right).
But this contradicts the hypothesis on $T$.

In both cases, we have a contradiction.
This completes the proof of Claim~\ref{clm:no_leaf}.
\end{proof}
Let $ab \in E(G)$ with $d(a)=2$ and $d(b)=1$ if it exists; otherwise, say $a,b \notin V(G)$.

Let $C$ be a component of $G[B_r[D]\setminus\{b\}]$ and let $D_C := D \cap V(C)$. Fix an arbitrary vertex $y \in D_C$. For $U \subseteq V(G)$, let $B_r^C[U] := B_r[U] \cap C.$

We will construct a set $X=X_C$ iteratively as follows.
Initialize $X:=\{y\}$ and $B:=B^C_4[y]$.
At each step, if $D_C\setminus B\neq \emptyset$, choose $(v,u) \in X\times (D_C\setminus B)$ such that $d_G(v,u)$ is minimal.
In this case, add $u$ to $X$ and add elements of $B_4^C[u]$ to $B$.
Repeat until $D_C\setminus B = \emptyset$, in which case we terminate the process.

Let $\le$ be a total order on $X$.
For each $x \in X$, define $S_x$ to be the set of all $v\in V(C)$ such that for all $z \in X$, either $d_G(x,v)<d_G(z,v)$, or $d_G(x,v) = d_G(z,v)$ and $x\le z$.
Observe that $\{S_x:x\in X\}$ is a partition of $V(C)$. That is, $x \in X$ is a closest point to $v$ in $X$, and if there are multiple such points then ties are broken by the total order.

\begin{claim} \label{clm:Sx_ball}
For all $x\in X$, $B_2[x]\setminus\{b\}\subseteq S_x \subseteq B_{r+4}[x]$.
\end{claim}
\begin{proof}[Proof of Claim]
To show $B_2[x]\setminus\{b\}\subseteq S_x$, take any $v\in B_2[x]\setminus\{b\}$.
Note that $v \in B_r[D]$,as $x \in D$ and $r \ge 2$, so in particular $v \in C$.
For all $z \in X\setminus \{x\}$, we have $z\notin B_4[x]$ or equivalently $x\notin B_4[z]$ by the construction of $X$, so $d_G(x,z) > 4$.
By the triangle inequality, $d_G(x,z)\le d_G(x,v)+d_G(v,z)$, so $d_G(v,z) \ge d_G(x,z)-d_G(x,v)>4-2=2$.
Thus, $x$ is the strictly closest element of $X$ to $v$, so $v\in S_x$.

To show $S_x\subseteq B_{r+4}[x]$, take any $v\in S_x$. We will show that there is some $x' \in X$ such that $d(x',v) \le r+4$. This will suffice, as by definition of $S_x$, a closest point in $X$ to $v$ is $x$. Let $u \in D$ be chosen so that $d(u,v)$ is minimal. Since $v \in B_r[D]$, $d(u,v) \le r$. If $d(u,v) = d(u,x)$, then we are done. Otherwise, $u \notin X$, as $x$ is a closest point in $X$ to $v$. 
So $u \notin X$, so that way $X$ was constructed implies that $d(v,x') \le 4$ for some $x' \in X$. In particular, $d(x',v) \le d(x',v) + d(v,u) \le r+4$, as required.
Thus, $v\in B_{r+4}[x]$.
The claim follows.
\end{proof}

\begin{obs}
    \label{clm:Sx_connected}
For all $x \in X$, $G[S_x]$ is connected.
\end{obs}

Take any vertex $x \in X$.
By Lemma~\ref{lem:spanning_tree_dist}, let $Q_x$ be a spanning tree of $G_x:=G[S_x]$ (which is connected by Claim~\ref{clm:Sx_connected}) such that for all $v\in V(G_x)$, $d_{Q_x}(x,v) = d_{G_x}(x,v)$.
Let $L:=\{v \in V(Q_x):d_G(v)=1\}$, $Q_x':=Q_x-L$, and $G_x':=G_x-L$.
Note that $Q_x$, $Q_x'$, $G_x$, and $G_x'$ are all connected, since all vertices in $Q_x$ have a path to $x$, and removing leaves cannot disconnect graphs.
Let $x$ be the root of $Q_x$ and $Q_x'$.
For each $v \in V(Q_x)$, let $T_v$ be the subtree of $Q_x$ rooted at $v$, that is, the subtree of $Q_x$ containing $v$ and its descendants, defined to be any $u\in V(Q_x)$ such that the $u$-$x$ path in $Q_x$ contains $v$.

For each $v \in V(Q_x')$, let $h(v)$ be the height of $v$ in $Q_x'$, i.e., the maximum possible length of a path $v_0,\ldots,v_\ell$, where $v_0=v$ and $v_i$ is a child of $v_{i-1}$ in $Q_x'$ for all $1\le i \le \ell$.
For each $v \in V(Q_x')$, let $f(v):= d_G(v)-3$ if $v$ has a child $u$ in $Q_x$ with $u \in L$, and let $f(v):= d_G(v)-2$ otherwise.
Note that $f(v)$ is always non-negative by the definition of $Q_x'$ and by Claims~\ref{clm:one_leaf} and~\ref{clm:no_leaf}.
Let $T_v'$ be the subtree of $Q_x'$ rooted at $v$.
Let $F(v):= \sum_{u \in V(T_v')}f(u)$.
By Claim~\ref{clm:one_leaf}, in $G_x$, every vertex has at most one neighbour that is in $L$, and every vertex in $L$ has exactly one neighbour that is in $V(Q_x')$, so $F(v) = \sum_{u\in V(T_v)}(d_G(u) - 2)$.

\begin{claim} \label{clm:avg_degree_contradiction}
We have $F(x)\ge 1$.
\end{claim}
\begin{proof}[Proof of Claim]

Suppose otherwise.
Then $f(v)=0$, as $f(v)$ is non-negative, for all $v\in V(Q_x')$, so $V(Q_x)$ can be partitioned into pairs $\{u,v\}$ with $uv \in E(G_x)$ and $\{d_G(u),d_G(v)\}=\{1,3\}$, along with singletons $\{w\}$ with $d_G(w)=2$.
Thus, since $\Delta(G_x') = 2$ and $G_x'$ is connected, $G_x'$ is either a cycle or a path.

If $G_x'$ is a cycle, then since every vertex in $G_x'$ has all of its neighbours in $G_x$, we have that $G_x$ is a connected component of $G$ (recall that $G_x'$ is obtained from $G_x$ by removing leaves).
Since $d_G(x) \ge 3$, we have $|V(G_x)|\ge 4$, so there exists $uv \in E(\overline{G_x})$ such that $d_G(u),d_G(v) \le 2$.
Thus, by Claim~\ref{clm:one_leaf}, $\Delta(G_x+uv) \le 3$, and so $G_x+uv$ can be properly $4$-coloured, and so there exists a rainbow $T$-free proper colouring of $G+uv$, since $4<|E(T)|$.
This contradicts $G \in \Prsat(n,T)$.

Suppose $G_x'$ is a path $v_{-k},\ldots,v_\ell$, where $v_0=x$.
By the assumption that $F(x)=0$ and the properties of $G_x$, $v_{-k}$ and $v_\ell$ must be degree-$2$ vertices in $G$ with exactly one neighbour not in $G_x$ each, say $v_{-k-1}$ and $v_{\ell+1}$, respectively.
Let $\phi$ be a rainbow $T$-free proper colouring of $G$.
If $a \in \{v_{-1},v_1\}$, assume without loss of generality that $a=v_1$, and let $c_1:= \phi(v_{-3}v_{-2})$ and let $c_2\neq c_1$; otherwise, let $c_1:=\phi(v_{-3}v_{-2})$ and $c_2:=\phi(v_{2}v_{3})$.
(This definition is possible since $B_2[x]\setminus\{b\}\subseteq S_x$ by Claim~\ref{clm:Sx_ball}.)
We will recolour $G$ with a colouring $\phi'$ and add an edge $e$ such that $G+e$ is still rainbow $T$-free.

If $c_1\neq c_2$, define $(c_3,c_4):=(c_1,c_2)$; otherwise, let $c_3,c_4\neq c_1$ be distinct colours.
If $c_1\neq c_2$, let $c_5\notin\{c_1,c_2\}$; otherwise, let $c_5 := c_1$.
In the new colouring $\phi'$, recolour the path $v_{-2},v_{-1},v_0,v_1,v_2$ with alternating colours $c_4,c_3$.
Recolour all edges $v_iu\in E(G_x)$ with $-2\le i \le 2$ and $u \in L$ using the colour $c_5$, which is possible as each vertex of $G_x$ has degree at most 3 in $G$.
Leave all other edge colours unaltered.
Observe that $\phi'$ is a rainbow $T$-free proper colouring, since any tree subgraph of $G$ with diameter $\ge 4$ that contains a recoloured edge either has a repeated colour or has a degree-$2$ vertex with a degree-$1$ neighbour.

\begin{figure}
    \centering
    \includegraphics[width=\textwidth]{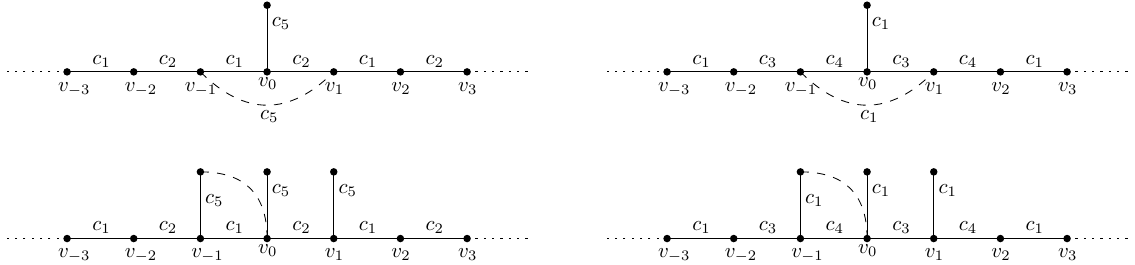}
    \caption{Examples of cases from Claim~\ref{clm:avg_degree_contradiction} when $r=2$ and $G_x'$ is a path $v_{-2},\ldots,v_2$. In all cases, an edge can be added to $(G,\phi)$ in some colour such that no rainbow copy of $T$ is created.}
    \label{fig:cat_converse_avg_degree_contradiction}
\end{figure}

In the new colouring $\phi'$, recolour the path $v_0,\ldots,v_k$ with alternating colours $c_4,c_3$.
Recolour all edges in $G_x$ that intersect $L$ with the colour $c_5$, which is possible as each vertex of $G'$ has degree at most 3 in $G$.
Leave all other edge colours unaltered.
Observe that $\phi'$ is a rainbow $T$-free proper colouring, since any tree subgraph of $G$ with diameter $\ge 4$ that contains a recoloured edge either has a repeated colour or has a degree-$2$ vertex with a degree-$1$ neighbour.

If $d_G(v_{-1}) = d_G(v_{1}) = 2$, add the edge $e:=v_{-1}v_{1}$ to $(G,\phi')$ in colour $c_5$.
Otherwise, let $u \in \{v_{-1},v_{1}\}$ with $d_G(u)=3$, let $w$ be the leaf neighbour of $u$, and add the edge $e:=xw$ to $(G,\phi')$ in a colour not in $\{c_3,c_4,c_5\}$.
Observe that this extension of $\phi'$ is a rainbow $T$-free proper colouring, since any tree subgraph of $G+e$ containing $e$ of order $|V(T)|$ either has a repeated colour or has a degree-$2$ vertex with a degree-$1$ neighbour (see Figure~\ref{fig:cat_converse_avg_degree_contradiction}).
This contradicts $G\in \Prsat(n,T)$.

In all cases, we obtain a contradiction.
This completes the proof of Claim~\ref{clm:avg_degree_contradiction}.
\end{proof}

\begin{claim} \label{clm:small_tree}
For all $v \in V(Q_x')$, we have 
\[|V(T_v')|\le \begin{cases}
    (F(v)+1)h(v)+1, & v \neq x, \\
    (F(v)+2)h(v)+1, & v=x.
\end{cases}\]
\end{claim}
\begin{proof}[Proof of Claim]
Use induction on $h(v)$.
When $h(v)=0$, we have $|V(T_v')| = |\{v\}|\le 1$.

Suppose $h(v)>0$.
Let $v_1,\ldots,v_k$ be the children of $v$ in $Q_x'$; note that $k\le f(v)+1$ if $v$ has a parent in $Q_x'$ (i.e., $v\neq x$), and $k\le f(v)+2$ otherwise.
Now, $F(v) = \sum_{i=1}^k F(v_i)+f(v)$, so by the induction hypothesis,
\begingroup
\allowdisplaybreaks
\begin{align*}
    |V(T_v)| &\le \sum_{i=1}^k|V(T_{v_i})|+1 \\
    &\le \sum_{i=1}^k((F(v_i)+1)h(v_i)+1)+1 \\
    &\le (h(v)-1)\sum_{i=1}^k F(v_i)+k(h(v)-1)+k+1 \\
    &\le \begin{cases}
        (h(v)-1)(F(v)-f(v))+(f(v)+1)h(v)+1, & v\neq x, \\
        (h(v)-1)(F(v)-f(v))+(f(v)+2)h(v)+1, & v=x \\
    \end{cases} \\
    &= \begin{cases}
        (F(v)+1)h(v)-(F(v)-f(v))+1, & v\neq x, \\
        (F(v)+2)h(v)-(F(v)-f(v))+1, & v=x
    \end{cases} \\
    &\le \begin{cases}
    (F(v)+1)h(v)+1, & v \neq x, \\
    (F(v)+2)h(v)+1, & v=x.
\end{cases}
\end{align*}
\endgroup
Therefore, the claim follows by induction.
\end{proof}
By Claim~\ref{clm:one_leaf}, for all $v\in V(Q_x')$, we have $|N_{G_x}(v)\cap L|\le 1$.
Further, for all $v \in L$, we have $|N_{G_x}(v)\cap V(Q_x')| = 1$.

By Claim~\ref{clm:Sx_ball}, $h(x)\le r+4$.
So by Claim~\ref{clm:small_tree},
\[|S_x| = |V(Q_x')|+|L| \le 2|V(Q_x')| = 2|V(T_x')| \le 2((F(x)+2)h(x)+1) \le 2((F(x)+2)(r+4)+1). \]
For any subset $\emptyset \neq S\subseteq V(G)$, define $d_{\mr{av}}(S):= \frac{\sum_{v\in S}d_G(v)}{|S|}$. Thus, since $F(x) = \sum_{v\in V(S_x)}(d_G(v)-2)$, we have
\[ d_{\mr{av}}(S_x) \ge 2+\frac{F(x)}{|S_x|}\ge 2+\frac{F(x)}{2((F(x)+2)(r+4)+1)}. \] 
Now, the function $g:\bb{R}^+\to \bb{R}$ defined by $g(z):= \frac{z}{2((z+2)(r+4)+1)}$ has $\frac{d g}{d z} = \frac{2r+9}{2((z+2)(r+4)+1)^2}>0$, so $g$ is increasing.
Thus, since $F(x)\ge 1$ by Claim~\ref{clm:avg_degree_contradiction}, we have 
\begin{equation} \label{eqn:cat_converse}
    d_{\mr{av}}(S_x) \ge 2+g(1) = 2+\frac{1}{6r+26}.
\end{equation} 
Let $\ca{C}$ be the set of components of $G[B_r[D]\setminus\{b\}]$; by Claim~\ref{clm:cat_converse_not_far}, $|\bigcup_{C \in \ca{C}}V(C)| \ge n-4$.
For each $C\in \ca{C}$, $\{S_x:x\in X_C\}$ is a partition of $V(C)$.
So by (\ref{eqn:cat_converse}), since $x$ was taken to be an arbitrary element of $X=X_C$, where $C$ is an arbitrary element of $\ca{C}$, we have
$$\prsat(n,T) \ge \frac{d_{\mr{av}}(V(G))}{2}n \ge \frac12 d_{\mr{av}}\left(\bigcup_{C\in \ca{C}}\bigcup_{x\in X_C}S_x\right)(n-4)  \ge \left(1+\frac{1}{12r+52}\right)n+O(1),$$
as desired.
\end{proof}
\begin{remark}
Theorem~\ref{thm:cat_converse} gives the only known lower bound on $\prsat(n,T)$ (beyond the somewhat trivial bound $\prsat(n,T)\ge \ff{n}{2}$) when $T$ is a caterpillar with central path $v_1,v_2,v_3$, where $d(v_1),d(v_3)\ge 3$ and $d(v_2)=2$.
In this case, because $\delta_2(T)=2$ and $T$ is neither a broom nor has diameter $\ge 5$, previous results do not apply.
\end{remark}

\section{Subdivided stars and double stars}\label{sec:stars}
\subsection{Subdivided stars}
For $k \ge 4$, let $T_k^*$ be the tree on $k$ vertices obtained by subdividing one edge of the star $K_{1,k-2}$.
 
The goal of this subsection is to prove the following theorem.
\begin{restatable}{theorem}{subdividedstar} \label{thm:Tkast rainbow sat}
For all $k \ge 4$ and $n \ge k+3$, 
\[\prsat(n,T_k^*) = \sat(n,T_{k+1}^*) = n-\ff{n+k-1}{k+1}.\]
\end{restatable}

This yields the following corollary.
\begin{corollary}
For all $n\ge 7$,
\[ \prsat(n,P_4) = n-\ff{n+3}{5}= \begin{cases}
    \cf{4n}{5}, & n \equiv 1 \pmod{5}, \\
    \ff{4n}{5}, & \text{otherwise}.
\end{cases} \]
\end{corollary}

We begin with the following lemmas.

\begin{lemma} \label{lem:Tkast}
Let $k \ge 4$, and let $G$ be a graph with a component $H \cong K_{1,k}$ and no isolated vertices.
Then for any non-edge $e \in E(\overline{G})$ incident to a vertex in $H$, $G+e$ contains a copy of $T_{k+1}^\ast$.
\end{lemma}
\begin{proof}
Let $V(H) = \{v,v_1,v_2,\ldots,v_k\}$, and let $E(H) = \{vv_i: i \in [k]\}$.
Take any $xy \in E(\overline{G})$ with $x \in V(H)$.

If $y \in V(H)$, then since there are no non-edges in $H$ containing the vertex $v$, we have $x,y \in \{v_1,\ldots,v_k\}$; say without loss of generality that $x=v_1$ and $y=v_2$.
Then the subgraph $T$ of $G+xy$ with $V(T) = V(H)$ and $E(T) = \{v_1v_2,vv_1\}\cup \{vv_3,\ldots,vv_k\}$ is a copy of $T_{k+1}^\ast$ (see Figure \ref{fig:t5ast}, left).

If $x \in \{v_1,\ldots,v_k\}$ and $y \notin V(H)$, then without loss of generality, $x = v_1$.
The subgraph $T$ of $G+xy$ with $V(T) = \{y,v,v_1,\ldots,v_{k-1}\}$ and $E(T) = \{v_1y,vv_1,vv_2,\ldots, vv_{k-1}\}$ is a copy of $T_{k+1}^\ast$ (see Figure \ref{fig:t5ast}, centre).

If $x = v$ and $y \notin V(H)$, then since $G$ has no isolated vertices, there exists $z \in N(y)$.
Because $y$ is in a separate component of $G$ from $H$, $z \notin V(H)$.
The subgraph $T$ of $G+xy$ with $V(T) = \{z,y,v,v_1,v_2,\ldots,v_{k-2}\}$ and $E(T) = \{zy,yv,vv_1,vv_2,\ldots,vv_{k-2}\}$ is a copy of $T_{k+1}^\ast$ (see Figure \ref{fig:t5ast}, right).

In all cases, $G+xy$ contains a copy of $T_{k+1}^\ast$, so we are done.
\begin{figure}
    \centering
    \includegraphics{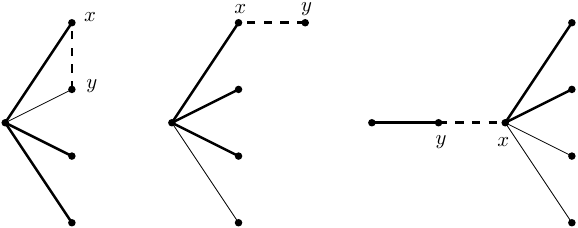}
    \caption{Copy of $T_k^\ast$ in $G+xy$ from Lemma \ref{lem:Tkast} in the case $k=5$}
    \label{fig:t5ast}
\end{figure}
\end{proof}

\begin{lemma} \label{lem:kTree_not_Tkast_sat}
For all $k\ge 4$ and $3 \le m \le k$, for any forest $F$ on $m$ vertices, there exists $e \in E(\overline{F})$ and a rainbow $T_k^\ast$-free proper colouring of $F+e$.
\end{lemma}
\begin{proof}
Let $F$ be a forest on $m$ vertices.
If $m < k$, then for any $e \in E(\overline{F})$, $F+e$ does not contain a copy of $T_k^\ast$, so let $m = k$.
Let $v \in V(F)$ with $d(v) = \Delta(F)$, and let $N(v) = \{v_1,\ldots,v_{\Delta(F)}\}$.

If $\Delta(F) = k-1$, then $F$ is a star, so colouring $F$ with a proper colouring and adding the non-edge $v_1v_2$ in the colour of the edge $vv_3$ results in a rainbow $T_k^\ast$-free proper colouring of $F$, since the edge $v_1v_2$ and the leaf $v_3$ must both be present in any copy of $T_k^\ast$ in $T+v_1v_2$.

If $\Delta(F) = k-2$ and $F$ is disconnected, then there exists a tree component $T$ of $F$ with $\Delta(T) = k-2$, $E(\overline{T}) \neq \emptyset$, and $|V(T)| < k$, so adding a non-edge in $T$ to $F$ does not create a copy of $T_k^\ast$.

If $\Delta(F) = k-2$ and $F$ is a tree, then $F \cong T_k^\ast$, so if $k = 4$, adding an edge between two leaves creates a copy of $C_4$, which can be properly $2$-coloured; otherwise, adding an edge between two leaves $v_i,v_j \in N(v)$ does not create a new copy of $T_k^\ast$ in $F$ ($vv_i$ and $vv_j$ must both be present in a copy $T$ of $T_k^\ast$ in $F$, else $\Delta(T) \le k-3 < \Delta(T_k^\ast)$).

If $\Delta(F) \le k-3$, then if $k = 4$, this gives $m \le 2 < k$, contradicting the assumption $m=k$; otherwise, for two vertices $u,w \in V(F)$ of degree at most $1$, $\Delta(F+uw) \le \max\{\Delta(F),2\} = \Delta(F) \le k-3 < k-2 = \Delta(T_k^\ast)$, so $F+uw$ does not contain a copy of $T_k^\ast$.

Therefore, in all cases, there exists $e \in E(\overline{T})$ and a $T_k^\ast$-free proper colouring of $T+e$, whose restriction to $E(F)\cup \{e\}$ is a rainbow $T_k^\ast$-free proper colouring of $F+e$, as desired.
\end{proof}

In order to prove Theorem~\ref{thm:Tkast rainbow sat} we will apply the following result. Let $\Sat(n,H)$ be the set of $n$-vertex $H$-saturated graphs with $\sat(n,H)$ edges.

\begin{restatable}[J. Faudree, R. Faudree, Gould, Jacobson~\cite{FaudreeFaudreeGouldJacobson2009}]{theorem}{subdividedstarsat} \label{thm:T_k^ast}
For any $k\ge 5$ and $n \ge k+2$, $\sat(n,T_k^*) = n-\lfloor (n+k-2)/k \rfloor$ and $\Sat(n,T_k^*)$ contains a star forest.
\end{restatable}

We are now ready to prove Theorem~\ref{thm:Tkast rainbow sat}.
\begin{proof}[Proof of Theorem~\ref{thm:Tkast rainbow sat}]
In order to prove that $\prsat(n,T_k^\ast) \le \sat(n,T_{k+1}^\ast)$, let $G \in \Sat(n,T_{k+1}^\ast)$ be a star forest, which exists by Theorem~\ref{thm:T_k^ast}.
Our goal is to show that $G$ is properly rainbow $T_k^\ast$-saturated.
First, note that $G$ does not contain a copy of $T_k^\ast$.

Consider $T_{k+1}^\ast$ with a proper colouring.
Let $v$ be the vertex of degree $k-1$, let $x$ be the vertex not adjacent to $v$, and let $y$ be the neighbour of $x$.
At most one vertex $u \in N(v)$ (not $y$) has $vu$ coloured with the colour of $xy$, so $V(T_{k+1}^\ast)\setminus \{u\}$ induces a rainbow copy of $T_k^\ast$.

Since $G \in \Sat(n,T_{k+1}^\ast)$, for all $e \in E(\overline{G})$, $G+e$ contains a copy of $T_{k+1}^\ast$, which contains a rainbow copy of $T_k^\ast$ in all proper colourings.
Therefore, $G$ is properly rainbow $T_k^\ast$-saturated, proving $\prsat(n,T_k^\ast) \le \sat(n,T_k^\ast)$.

Now we prove that $\prsat(n,T_k^\ast) \ge \sat(n,T_{k+1}^\ast)$.
For a graph $G$, let $c(G)$ be the number of components of $G$.
Take $G \in \Prsat(n,T_k^\ast)$ such that $c(G) = \min\{c(H):H \in \Prsat(n,T_k^\ast)\}$.
We will show that $G$ has few components (and thus many edges) by supposing otherwise and replacing components of $G$ with a star such that the resulting graph has fewer components than $G$, contradicting the minimality of $c(G)$.
\begin{claim} \label{clm:subdivided_star1}
    $G$ does not contain two components each with order at most $k$.
\end{claim}

\begin{proof}[Proof of Claim]
    Suppose otherwise for contradiction.
    Let $X$ and $Y$ be distinct components of $G$ such that $|V(X)|\le |V(Y)| \le k$ and $|V(X)|$ is minimal, i.e., $|V(X)| \le |V(C)|$ for any component $C$ of $G$.
    Let $x=|V(X)|$ and $y=|V(Y)|$.
    
    If $x+y < k$, then for some $v \in X$ and $u \in Y$, $(X\cup Y)+vu$ has fewer than $k$ vertices, so adding $vu$ to $G$ does not add a rainbow copy of $T_k^\ast$, a contradiction.
    If $x+y = k$, then by Lemma \ref{lem:kTree_not_Tkast_sat}, $X\cup Y$ is not a forest, i.e., it is cyclic.
    If $x+y \ge k+1$, then $3 \le \lceil\frac{k+1}{2}\rceil \le y \le k$, so by Lemma \ref{lem:kTree_not_Tkast_sat}, $Y$ is not a forest, i.e., it is cyclic.
    Thus in all cases, $X\cup Y$ is cyclic.
    
    Now, removing an edge from a cycle in $X \cup Y$ and adding an edge between $X$ and $Y$ makes $X \cup Y$ connected, meaning that $|E(X\cup Y)| \ge x+y-1$, which is the number of edges on a star on $x+y$ vertices.
    So the graph $G'$ obtained by replacing $X$ and $Y$ with a star on $x+y$ vertices has $|E(G')|\le |E(G)|$ and $c(G') < c(G)$.
    Note that $|V(X)|$ is minimal and $G$ has at most one isolated vertex (else adding an edge between two isolated vertices would not create a copy of $T_k^\ast$), so $G'$ has no isolated vertex, and so by Lemma \ref{lem:Tkast}, $G' \in \Prsat(n,T_k^\ast)$.
    Therefore, this contradicts minimality of $c(G)$.
\end{proof}
\begin{claim} \label{clm:subdivided_star2}
    $G$ does not contain an isolated vertex.
\end{claim}

\begin{proof}[Proof of Claim]
    Suppose otherwise for contradiction.
    Let $v$ be the isolated vertex in $G$.
    Let $C$ be a component of $G-v$.
    By the claim above, $|V(C)| \ge k+1$.
    If $C$ has a cycle, then replace $C$ and $v$ with a star $S$ on $|V(C)|+1$ vertices to create a new graph $G'$.
    $|E(C)|\ge |V(C)|$, which is the number of edges on a star on $|V(C)|+1$ vertices.
    Since $G'$ has no isolated vertex, by Lemma \ref{lem:Tkast}, for any non-edge $e \in E(\overline{G})$ incident to a vertex in $S$, $G+e$ contains a copy of $T_{k+1}^\ast$, which contains a rainbow copy of $T_k^\ast$ in all proper colourings.
    Since $G \in \Prsat(n,T_k^\ast)$, adding any other non-edge to $G'$ results in a rainbow copy of $T_k^\ast$ in all proper colourings.
    So $G'$ is properly rainbow $T_k^\ast$-saturated.
    $|E(G')| \le |E(G)|$, so $G' \in \Prsat(n,T_k^\ast)$.
    But $c(G')<c(G)$, contradicting minimality of $c(G)$.
    
    So we may suppose that $C$ is a tree.
    But then if $C$ is a non-star with $\Delta(C)\ge k-1$, then it contains a copy of $T_{k+1}^\ast$ and thus forces a rainbow copy of $T_k^\ast$.
    If $C$ is a star, then if $u$ is the universal vertex of $C$, adding the edge $uv$ creates another star, which does not contain $T_k^\ast$.
    Otherwise, let $t$ be a leaf of $C$, and add the edge $vt$ to $G$.
    The resulting component $C'$ on vertices $V(C)\cup \{v\}$ is a tree with $\Delta(C') = \Delta(C) \le k-2$, so it can be properly coloured with $k-2 < |E(T_k^\ast)|$ colours.
    Therefore, all cases result in a contradiction.
\end{proof}
By the above two claims, every component of $G$ has order at least $2$ and at most one component of $G$ has order at most $k$.
So $n \ge (k+1)(c(G)-1)+2$, so rearranging, $c(G) \le \lfloor \frac{n-2}{k+1} \rfloor + 1 = \lfloor \frac{n-2+(k+1)}{k+1}\rfloor$.
$|E(G)| \ge n-c(G)$ because each component $C$ of $G$ has at least $|V(C)|-1$ edges.
Combining these inequalities yields $|E(G)| \ge n-\lfloor \frac{n-2+(k+1)}{k+1}\rfloor = \sat(n,T_{k+1}^\ast)$ by Theorem \ref{thm:T_k^ast}.
\end{proof}

\subsection{Double stars}
The \emph{double star} $S_{t+1,s+1}$ is constructed by appending $t$ and $s$ pendant vertices, respectively, to each vertex in $K_2$.
J. Faudree, R. Faudree, Gould, and Jacobson~\cite{FaudreeFaudreeGouldJacobson2009} prove the following bounds on the (standard) saturation number for double stars.

\begin{restatable}[J. Faudree, R. Faudree, Gould, Jacobson~\cite{FaudreeFaudreeGouldJacobson2009}]{theorem}{doublestarsat} \label{thm:faudree_double_star_sat}
For $t \ge s \ge 1$ and $n \ge (s+1)^3$,
\[\frac{s}{2}n \le \sat(n,S_{t+1,s+1}) \le \begin{cases}
    \frac{s}{2}n+\frac{t(t+2)}{2}, & t=s, \\
    \frac{s+1}{2}n-\frac{s^2+8}{8}, & t>s.
\end{cases} \]
\end{restatable}
We have the following result that improves the upper bound Theorem~\ref{thm:faudree_double_star_sat} in the case $t>s$ and provides analogous bounds on the proper rainbow saturation number for double stars.
\begin{restatable}{theorem}{doublestarsatimproved} \label{thm:double_star_sat_improved}
Let $t \ge s \ge 1$.
Then for all $n$,
\begin{enumerate}
    \item $\sat(n,S_{t+1,s+1}) \le \frac{(\cf{t+1}{s}+1)s}{(\cf{t+1}{s}+1)s+1}\cdot \frac{s+1}{2}n+O(1)$.
    \item $\frac{s}{2}n \le \prsat(n,S_{t+1,s+1}) \le \frac{(\cf{t+1}{s}+2)s}{(\cf{t+1}{s}+2)s+1}\cdot \frac{s+1}{2}n+O(1)$.
\end{enumerate}
\end{restatable}
In Theorem~\ref{thm:double_star_sat_improved}(2), when $s=1$, we have $S_{t+1,s+1}=S_{t+1,2}=T_{t+3}^*$, and the upper bound $\prsat(n,S_{t+1,2}^*)\le \frac{t+3}{t+4}n+O(1)$ recovers the value of $\prsat(n,T_{t+3}^*)$ from Theorem~\ref{thm:Tkast rainbow sat} up to the addition of a constant term.

For a non-regular graph $H$, let $\delta_2(H)$ denote the second smallest degree of $H$, i.e., $\delta_2(H) = \min(\{d(v):v\in V(H)\}\setminus\{\delta(H)\})$.
We have the following result by J. Faudree, R. Faudree, Gould, and Jacobson~\cite{FaudreeFaudreeGouldJacobson2009} bounding the standard saturation number for trees based on the second smallest degree.
\begin{theorem}[J. Faudree, R. Faudree, Gould, Jacobson~\cite{FaudreeFaudreeGouldJacobson2009}] \label{thm:Faudree_tree_gen_lower}
If $T$ is a non-star tree with $|V(T)|\ge 5$, then for all $n\ge (\delta_2(T)-1)^3$, $\sat(n,T)\ge \frac{\delta_2(T)-1}{2}n$.
\end{theorem}
For graphs $G$ and $H$, say $G$ is \emph{$H$-semi-saturated} if adding any edge to $G$ creates a copy of $H$ (but $G$ is not necessarily $H$-free).
Let $\ssat(n,H)$ denote the minimum number of edges in an $H$-semi-saturated graph on $n$ vertices.
Note that any $H$-saturated graph and any properly rainbow $H$-saturated graph is necessarily $H$-semi-saturated, so we have $\sat(n,H)\ge \ssat(n,H)$ and $\prsat(n,H)\ge \ssat(n,H)$. 

We observe that the proof of Theorem~\ref{thm:Faudree_tree_gen_lower} gives a lower bound on the number of edges in a $T$-saturated graph $G$ without using the property that $G$ is $T$-free.
Thus, Theorem~\ref{thm:Faudree_tree_gen_lower} can be generalized to a lower bound on $\ssat(n,T)$, which in turn gives a lower bound on $\prsat(n,T)$.
\begin{proposition} \label{prop:Faudree_tree_gen_lower_prsat}
If $T$ is a non-star tree with $|V(T)|\ge 5$, then for all $n\ge (\delta_2(T)-1)^3$, $\prsat(n,T) \ge \ssat(n,T)\ge \frac{\delta_2(T)-1}{2}n$.
\end{proposition}

Since $\delta_2(S_{t+1,s+1})=s+1$, the lower bound in Theorem~\ref{thm:double_star_sat_improved}(2) follows from Proposition~\ref{prop:Faudree_tree_gen_lower_prsat}.
In order to prove Theorem~\ref{thm:double_star_sat_improved}, we have the following construction.
\begin{construction} \label{cns:double_star}
Let $t\ge s \ge 1$, let $m=m(t,s):=\cf{t+1}{s}+1$, and let $G$ be a disjoint union of graphs of the form $K_1+kK_s$, where $k \ge m$.
Then $G$ is $S_{t+1,s+1}$-saturated.
\end{construction}
\begin{proof}
$G$ is $S_{t+1,s+1}$-free, since each component has only one vertex of degree at least $s+1$, and $S_{t+1,s+1}$ has two such vertices.
For all $uv \in E(\overline{G})$, one of $u$ or $v$ is adjacent in $G+uv$ to a vertex $w$ with $d_G(w) \ge ms$; assume this applies to $u$.
Then $G+uv$ contains a copy $S$ of $S_{t+1,s+1}$, where $w$ and $u$ are the vertices of degree $t+1$ and $s+1$ in $S$, respectively.
\end{proof}

We are now ready to prove Theorem~\ref{thm:double_star_sat_improved}.
\begin{proof}[Proof of Theorem~\ref{thm:double_star_sat_improved}]
The lower bound in (2) follows from Theorem~\ref{thm:Faudree_tree_gen_lower}, so it suffices to prove the upper bounds in (1) and (2).
As in Construction~\ref{cns:double_star}, let $m=m(t,s):=\cf{t+1}{s}+1$.
For any $d \in \bb{N}$, $d$ divides $ms+1$ and $(m+1)s+1$ only if it divides the integer linear combination $(m+1)(ms+1)-m((m+1)s+1)=1$.
Thus, $\gcd(ms+1,(m+1)s+1)=1$.
Because we have taken $n$ to be large enough, write $n=a(ms+1)+b((m+1)s+1)$ for non-negative integers $a$ and $b$ with $b$ as small as possible; note that $b=O(1)$.

Let $G=G(n,t,s):=a(K_1+mK_s)\cup b(K_1+(m+1)K_s)$.
By Construction~\ref{cns:double_star}, $G$ is properly rainbow $S_{t+1,s+1}$-saturated, so
\[ \sat(n,S_{t+1,s+1})\le |E(G)| = \frac{ms}{ms+1}\cdot \frac{s+1}{2}n+O(1), \]
proving (1).

To prove (2), note that $G(n,t+s,s)$ is $S_{t+1,s+1}$-free and $S_{t+s+1,s+1}$-saturated, and $S_{t+s+1,s+1}$ contains a rainbow copy of $S_{t+1,s+1}$ in any proper colouring, so $G(n,t+s,s)$ is properly rainbow $S_{t+1,s+1}$-saturated.
Therefore,
\[ \prsat(n,S_{t+1,s+1})\le |E(G(n,t+s,s))| = \frac{m(t+s,s)s}{m(t+s,s)s+1}\cdot \frac{s+1}{2}n+O(1), \]
as desired.
\end{proof}

\section{Concluding remarks}\label{sec:concl}

In this paper, we have characterized the asymptotic behaviour of $\prsat(n,T)$ whenever $T$ is a caterpillar with central path $v_1,\ldots,v_\ell$ such that $d(v_\ell)=2$ and either $\ell \neq 3$ or $d(v_2)=2$.
While we have not mentioned the case when $T$ is a star, in fact, a star $K_{1,k}$ is rainbow in any proper colouring, so a graph $G$ is properly rainbow $K_{1,k}$-saturated if and only if it is $K_{1,k}$-saturated.
Since the number $\sat(n,K_{1,k})$ has been determined exactly by K\'aszonyi and Tuza~\cite{Kaszonyi1986}, we have the following.
\begin{proposition}[see~\cite{Kaszonyi1986}]
    For all $k \in \bb{N}$ and $n\ge k+1$, 
    \[\prsat(n,K_{1,k}) = \sat(n,K_{1,k}) = \begin{cases}
        \frac{k-1}{2}n-\frac12 \ff{k^2}{4}, & n \ge k+\ff{k}{2}, \\
        \binom{k}{2}+\binom{n-k}{2}, & k+1\le n \le k+\ff{k}{2}.
    \end{cases}\]
\end{proposition}

A natural next step is extending these results to further families of trees.
\begin{problem} \label{prb:caterpillar}
    Determine $\liminf_{n\to \infty}\frac{\prsat(n,T)}{n}$ and $\limsup_{n\to \infty}\frac{\prsat(n,T)}{n}$ for arbitrary trees $T$.
\end{problem}

Proposition~\ref{prop:Faudree_tree_gen_lower_prsat}, which is the most general known lower bound on the proper rainbow saturation number for trees, bounds $\prsat(n,T)$ based on the second smallest degree $\delta_2(T)$.
By Theorem~\ref{thm:faudree_double_star_sat}, the corresponding bound on $\sat(n,T)$ in Theorem~\ref{thm:Faudree_tree_gen_lower} is asymptotically tight when $T$ is a symmetric double star $S_{t+1,t+1}$.
However, the bound in Proposition~\ref{prop:Faudree_tree_gen_lower_prsat} on $\prsat(n,T)$ is not known to be asymptotically tight for any tree $T$: every tree $T$ for which asymptotically tight bounds on $\prsat(n,T)$ are known is a caterpillar with $\delta_2(T) =2$ and $\prsat(n,T)\ge \frac45 n+O(1)$.
It would be interesting to see how this situation changes when $\delta_2(T)>2$.
In particular, it would be interesting to improve Theorem~\ref{thm:double_star_sat_improved}(2) to asymptotically tight bounds on $\prsat(n,S_{t+1,s+1})$ for $t \ge s \ge 2$.

It would also be interesting to see if there exists a graph parameter which yields a closer approximation to $\prsat(n,T)$ than $\delta_2(T)$ does.

In the case of the classical saturation number, it is known for each $k \ge 4$ which trees $T$ of order $k$ achieve the minimum and maximum possible saturation number.
\begin{theorem}[J. Faudree, R. Faudree, Gould, Jacobson~\cite{FaudreeFaudreeGouldJacobson2009}; K\'aszonyi, Tuza~\cite{Kaszonyi1986}]
Let $T$ be a tree of order $k \ge 4$.
Then for all $n$ large enough, $\sat(n,T_k^*)< \sat(n,T) < \sat(n,K_{1,k-1})$.
\end{theorem}
We ask whether the analogous statement holds for the proper rainbow saturation number.
\begin{question}
\label{con:min_max_prsat_tree}
Let $T$ be a tree of order $k \ge 4$.
Is it the case that for all $n$ large enough, $\prsat(n,T_k^*)<\prsat(n,T)<\prsat(n,K_{1,k-1})$?
\end{question} 
By combining Theorem~\ref{thm:gen_path_lower_bound}, Theorem~\ref{thm:broom4}, Theorem~\ref{thm:cat_converse}, and Proposition~\ref{prop:Faudree_tree_gen_lower_prsat}, we are close to proving the lower bound $\prsat(n,T)>\prsat(n,T_k^*)$ in Question~\ref{con:min_max_prsat_tree}: the only remaining case is when $T$ is a non-broom with a vertex of degree $2$ which has a leaf neighbour and longest path $P_5$.

In Section~\ref{sec:stars} we discover a connection between the proper rainbow saturation number of certain trees, and the classical saturation number of related trees. It would be interesting to find further connections between these parameters. 

One particular direction of interest could be to note that by the proof of Theorem~\ref{thm:Tkast rainbow sat}, for all $k \ge 4$ and $n \ge k+2$, $\Prsat(n,T_k^*) \cap \Sat(n,T_{k+1}^*) \neq \emptyset$.
Also, the graph $G(n,t+s+1,s+1)$ from the proof of Theorem~\ref{thm:double_star_sat_improved} is both properly rainbow $S_{t+1,s+1}$-saturated and $S_{t+s+1,s+1}$-saturated (although it does not necessarily achieve $\prsat(n,S_{t+1,s+1})$ or $\sat(n,S_{t+1,s+1})$ edges).
Thus, we pose the following question.
\begin{question} \label{qst:prsat_sat_connection}
For which graphs $H$ does there exist a graph $H'$ such that for all $n$ large enough, $\Prsat(n,H) \cap \Sat(n,H') \neq \emptyset$?
\end{question}
The phenomenon in Question~\ref{qst:prsat_sat_connection} occurs when $H'$ is the ``smallest" graph that contains a rainbow copy of $H$ in any proper colouring, in the sense that for ``most" graphs $G$, $G$ contains $H'$ if and only if it contains a rainbow copy of $H$ in any proper colouring.

In~\cite{other}, we consider the proper rainbow saturation number of other families of graphs. In particular, we prove bounds for complete graphs, cycles, and complete bipartite graphs. We suggest several other directions for further research there.

\textbf{Note added before submission:} While preparing this article we became aware that Baker, Gomez-Leos, Halfpap, Heath, Martin, Miller, Parker, Pungello, Schwieder, and Veldt~\cite{emily} had simultaneously and independently proved the same bounds on $\prsat(n,P_k)$ as in Corollary~\ref{cor:path} when $k \ge 6$, and they also proved that $\prsat(n,P_5) = n+O(1)$, although the result that $\prsat(n,P_5)=n-1$ for infinitely many $n$ is unique to our paper.
Their lower bound can be adapted to a bound on $\prsat(n,T)$ for all trees $T$ with diameter at least $4$.
While their construction for the upper bound on $\prsat(n,P_k)$ for $k\ge 6$ is essentially identical to ours, their proof differs significantly from ours and offers a very detailed characterization of the colourings involved.

\bibliographystyle{amsplain}
\bibliography{RainbowSat}

\end{document}